\documentclass[a4paper,parskip,11pt,listof=totoc,bibliography=totoc, index=totoc,twoside]{scrartcl}

\newcommand{\CommonPath}{.}
\usepackage[english]{babel}
\usepackage{csquotes}

\usepackage{aligned-overset}
\usepackage[right=1.8cm,left=1.9cm]{geometry}
\usepackage{amsmath,amssymb,amsfonts,amsthm,mathtools}
\mathtoolsset{ showonlyrefs, showmanualtags }
\usepackage{dsfont}
\usepackage{enumitem}
\setlist[enumerate]{label=\roman*),ref=\roman*)}
\usepackage{derivative}
\usepackage{array} 
\newcolumntype{L}{>{$}l<{$}}
\usepackage{subfiles}

\usepackage[svgnames]{xcolor}
\usepackage{subcaption}
\usepackage{graphicx}
\usepackage{tikz}
\usetikzlibrary{arrows.meta}
\usetikzlibrary{shadows}
\usepackage{ifthen}
\usepackage{intcalc}
\usepackage{float}

\usepackage{hyperref}
\hypersetup{
    colorlinks,
    citecolor=black,
    filecolor=black,
    linkcolor=black,
    urlcolor=black,
    linktoc=all,  
}

\theoremstyle{plain}
\newtheorem{theorem}{Theorem}[section]

\newtheorem{lemma}[theorem]{Lemma}

\newtheorem{proposition}[theorem]{Proposition}

\newtheorem{corollary}[theorem]{Corollary}

\newtheorem{definition}[theorem]{Definition}

\theoremstyle{definition}
\newtheorem{notation}[theorem]{Notation}

\newtheorem{example}[theorem]{Example}

\newtheorem{remark}[theorem]{Remark}

\newcommand{\Jr}{J}
\newcommand{\Js}{\dot{J}}
\newcommand{\Ir}{I}
\newcommand{\Is}{\dot{I}}
\newcommand{\ar}{a}
\newcommand{\as}{\dot{a}}
\newcommand{\br}{b}
\newcommand{\bs}{\dot{b}}

\newcommand{\N}{\mathbb{N}}
\newcommand{\R}{\mathbb{R}}

\newcommand{\eps}{\varepsilon}
\newcommand{\vphi}{\varphi}
\newcommand{\indicator}[1]{\mathds{1}_{\left\{#1\right\}}}
\newcommand{\smallo}{\mathchoice
    {{\scriptstyle\mathcal{O}}}%
    {{\scriptstyle\mathcal{O}}}%
    {{\scriptscriptstyle\mathcal{O}}}%
    {\scalebox{.6}{$\scriptscriptstyle\mathcal{O}$}}%
  }
\newcommand{\bigO}{\mathcal{O}}

\DeclarePairedDelimiter{\norm}{\lVert}{\rVert}

\DeclareMathOperator{\sgn}{sgn}

\newcommand{\weakstar}{\overset{\ast}{\rightharpoonup}}
\newcommand{\loc}{\text{loc}}

\newcommand{\const}{\text{const}}
\newcommand{\dd}{\mathop{}\!\mathrm{d}}
\newcommand{\del}{\Delta\,}
\newcommand{\delr}{\Delta_r\,}
\newcommand{\dels}{\Delta_s\,}

\NewDocumentCommand{\summ}{O{1} O{\infty} }{\sum\limits_{m=#1}^{#2}}
\NewDocumentCommand{\sumrs}{O{1} O{\infty} O{r+s}}{\sum\limits_{#3=#1}^{#2}}
\NewDocumentCommand{\sumkl}{O{1} O{N} O{k+l}}{\sum\limits_{#3=#1}^{#2}}
\NewDocumentCommand{\sumRM}{O{\le} O{\le}}{\sum\limits_{r+s\ge 2}^{\substack{r#1R \\ s#2M}}}
\NewDocumentCommand{\sumgamma}{O{\Gamma}}{\sum\limits_{(r,s) \in #1}}
\NewDocumentCommand{\sumbar}{O{r} O{r+s} }{\sum\limits_{k=#1}^{#2} \cbar_{k,#2-k}}
\NewDocumentCommand{\sumhat}{O{r} O{r+s} }{\sum\limits_{k=#1}^{#2} \chat_{k,#2-k}}
\newcommand{\monr}{c_{1,0}}
\newcommand{\mons}{c_{0,1}}
\newcommand{\rs}{_{r,s}}
\newcommand{\kl}{_{k,l}}

\newcommand{\rps}{_{r+1,s}}

\newcommand{\rms}{_{r-1,s}}
\newcommand{\rsp}{_{r,s+1}}

\newcommand{\rsm}{_{r,s-1}}
\newcommand{\rpsp}{_{r+1,s+1}}

\newcommand{\Vzw}{V}
\newcommand{\Vzwn}{V^{N}}
\newcommand{\Vn}{V^{N}}
\newcommand{\D}{D}
\newcommand{\Ex}{\mathcal{E}}
\newcommand{\Dn}{D^{N}}
\newcommand{\Dnp}{D^{N+1}}

\newcommand{\cinfzw}{\bar{c}^{z,w}}
\newcommand{\cinfuv}{\bar{c}^{u,v}}
\newcommand{\cbar}{\bar{c}}
\newcommand{\cbargamma}{\bar{c}^\Gamma}
\newcommand{\Brohsigp}{B_{\rho,\sigma}^+}

\usepackage[backend=biber, style=numeric, giveninits=true, sorting=nyt, url=false, doi=false, eprint=true, isbn=false]{biblatex}
\KOMAoptions{abstract=true}
\usetikzlibrary{external}
\numberwithin{equation}{section}

\title{The Two-Component Becker--Döring System: Basic Properties}
\author{Jens Scholten}
\date{\today}

\begin{document}
  \maketitle
  
\begin{abstract}
  \noindent The two-component Becker--Döring model is a two-dimensional coagulation--fragmentation equation that describes the dynamics of clusters built from two different types of monomers growing and shrinking through monomer interaction. We show the existence and mass conservation of solutions and give a partial result regarding uniqueness. Furthermore, we prove under the assumption of detailed balance the entropy equation, followed by a detailed study of the equilibrium states and relative entropies. We can determine the minimum of all relative entropies under sequences of fixed Type I and Type II masses and fully describe the minimising sequences in the weak* topology.
\end{abstract} 
  \tableofcontents  
  
\section{Introduction}
\subsection{One-Component System}
The classical Becker--Döring equations \cite{beckerDoering} are considered to be amongst the simplest of coagulation--fragmentation equations. Particles are assumed to be fully characterised by the number of monomers from which they are composed and may be growing or shrinking, either by condensation or evaporation of a single monomer. For given initial conditions and under the assumption of mass conservation, we can then describe the evolution of the particle distributions. If we denote for any $i\in \N_{>0}$ with $c_i(t)$ the density of particles containing exactly $i$ monomers at time $t$, we can determine its time evolution via 
\begin{equation}\label{ocBD}
  \begin{cases}
    \odv{}{t} c_i(t) = - J_i(t) + J_{i-1}(t) \text{ if }i>1 \text{ and }\\
    \sum\limits_{n=1}^{\infty} n c_n(t) = \sum\limits_{n=1}^{\infty} n c_n(0),
  \end{cases}
\end{equation}
where $J_i(t)$ is the flux between particles of size $i$ and those of size $i+1$. When the rate at which a monomer attaches to a particle of size $i$ is specified as $a_i$ and the rate of detachment from a particle of size $i+1$ as $b_{i+1}$, we can describe 
\[ J_i(t) = a_i c_1 (t) c_i(t) - b_{i+1} c_{i+1}(t).\]
Furthermore, the conservation of mass is formally equivalent to 
\[ \odv{}{t} c_1(t) = -J_1 - \sum\limits_{n=1}^{\infty} J_n(t),\]
which is the formulation that was already introduced in the original study of \eqref{ocBD} by Penrose and Lebowitz \cite{penrose:first}.
System \eqref{ocBD} has received a lot of attention in the mathematical literature \cite{ball:refinedMaxPrinciple,canizo:EDEstimate,penrose:LSWConnection,slemrod:longTime}. Despite its apparent simplicity, many interesting phenomena can be found in \eqref{ocBD}, e.g. phase transitions \cite{penrose:foundations}, metastability \cite{penrose:metastability}, coarsening effects \cite{niethammer:LswLimit} or local exponential stability \cite{canizo:SpectralGab}.
The necessity of a phase transition for supersaturated initial conditions is easily seen. With the notation
\begin{equation}\label{ocDetailedBalance}
  Q_i \coloneqq \prod\limits_{k=1}^{i-1} \frac{a_k}{b_{k+1}},
\end{equation}
all steady states of \eqref{ocBD} are given by a one-parameter family, indexed by the monomer density $z>0$ via $\bar{c}^z_i \coloneqq z^i Q_i.$ Typical coefficients describing droplets in a vapour \cite{niethammer:LswLimit} or a binary alloy on a cubic lattice \cite{kalos:BinaryAlloyCoefficients} will behave like
\[ Q_i \approx \frac{1}{z_s^i} e^{-i^{\frac 2 3}} \text{ for } i\to\infty,\text{ with } 0<z_s <\infty.\]
Hence, there is a one-to-one correspondence from $0\le \rho \le \rho_s \coloneqq \sum\limits_{n=1}^{\infty} n  z_s^n Q_n < \infty$ to $0\le z_\rho \le z_s$ such that the steady state $z_\rho^i Q_i$ has mass $\rho$ and for $\rho>\rho_s$, there is no steady state attaining that mass. The long time behaviour was determined in the seminal paper \cite{penrose:foundations}, where it was shown to depend on the initial mass of the system $\rho \coloneqq \sum\limits_{n=1}^{\infty} n c_n(0)$ via
\begin{equation}\label{ocLongTime}
  c_i(t) \to 
  \begin{cases}
    z_\rho^i Q_i & \text{ if } \rho \le \rho_s\\
    z_s^i Q_i & \text{ else}.
  \end{cases}
\end{equation}
Despite the conservation of mass for all times $\rho(t) = \rho(t)$, if $\rho > \rho_s$, then mass will be lost in the limit $t=\infty$; the aforementioned phase transition. To prove \eqref{ocLongTime}, the following entropy structure of \eqref{ocBD} is used
\begin{equation*}
  \odv{}{t} \underbrace{\sum\limits_{n=1}^{\infty} c_n\left(\ln\left(\frac{c_n}{Q_n}\right) -1\right)}_{\eqqcolon V(c)}
  =-\sum\limits_{n=1}^{\infty} (a_n c_1c_n - b_{n+1}c_{n+1})\ln\left(\frac{a_n c_1 c_n}{b_{n+1} c_{n+1}}\right) \le 0.
\end{equation*}
In \cite{penrose:foundations} they show that the continuity of the entropy $V(c)$ implies that $c(t)$ converges to some steady state. The exact steady state is then determined by a maximum principle \cite{ball:refinedMaxPrinciple}. In the following we will prove, that such a phase transition must also occur in the two-component system (for appropriate coefficients). However, we do not fully characterise the limiting behaviour, which is considerably harder, because $V(c)$ will not be weak* continuous any longer.

Finally, we want to mention the excellent review \cite{hingant:Overview} and the references therein for a more detailed overview of the one-component Becker--Döring system.

\subsection{Two-Component System}
For $(r,s) \in \Omega \coloneqq \{(r,s) \in \N^2_0 \,|\, r+s\ge 1\}$, let us denote the concentration of clusters containing $r$ Type I monomers and $s$ Type II monomers as $c\rs.$ Then, the two-component Becker--Döring equation reads
\begin{equation}\label{mcBD}
  \begin{cases}
    \odv{}{t}c\rs  &= -\delr \Jr\rs - \dels \Js\rs\quad \text{for }r+s\ge2,\\
    \odv{}{t}c_{1,0} &= -\Jr_{1,0}-\Js_{1,0} - \sum\limits_{r+s=1}^\infty  \Jr\rs,\\
    \odv{}{t}c_{0,1} &= -\Jr_{0,1}-\Js_{0,1} - \sum\limits_{r+s=1}^\infty \Js\rs,
  \end{cases}
\end{equation}
where $\Jr\rs$ denotes the flux between particles of size $(r,s)$ and $(r+1,s)$ and $\Js\rs$ the flux between those of size $(r,s)$ and $(r,s+1).$ Furthermore, $\delr$ and $\dels$ are the discrete backwards derivative operators, which are defined for a sequence $(f\rs)_{(r,s)\in\Omega}$ via
\[ \delr f\rs \coloneqq f\rs - f\rms \text{ and }\dels f\rs \coloneqq f\rs - f\rsm, \text{ with the convention }f\rs = 0 \text{ for all }(r,s) \not\in \Omega.\]
As in the one-component case, the coagulation and fragmentation coefficients are assumed to be given
\begin{center}
  \begin{tabular}{| L | l | L |}
    \text{parameter} & description & \text{typical example}\\
    \hline
    \ar\rs \ge 0 &rate for the process $(r,s)+(1,0) \implies (r+1,s)$ &\quad(r+s)^{\frac 2 3}\\
    \as\rs \ge 0&rate for the process $(r,s)+(0,1) \implies (r,s+1)$&\quad(r+s)^{\frac 2 3}\\
    \br\rps \ge 0&rate for the process $(r+1,s) \implies (r,s)+(1,0)$ &\quad \ar\rs \frac{r+1}{r+s+1}\\
    \bs\rsp \ge 0&rate for the process $(r,s+1) \implies (r,s)+(0,1)$&\quad\as\rs\frac{s+1}{r+s+1}
  \end{tabular}
\end{center}
and allow to describe the fluxes as 
\begin{equation}\label{fluxes}
    \Jr\rs \coloneqq \ar\rs c_{1,0}c\rs  - \br\rps c\rps
    \quad\text{ and }\quad
    \Js\rs \coloneqq \as\rs c_{0,1}c\rs  - \bs\rsp c\rsp.
\end{equation}

The total Type I and Type II masses of the system---denoted $\rho$ and $\sigma$ respectively---are given as 
\begin{equation}\label{masses}
  \rho(c) = \sumrs r c\rs \text{ and }\sigma(c) = \sumrs s c\rs.
\end{equation}
If we multiply \eqref{mcBD} with a test-sequence $g\rs$ and sum from $2\le M$ to $N<\infty$, we obtain 
\begin{equation}\label{weakmcBDfinite}
  \begin{split}
    \odv{}{t} \sumrs[M][N] g\rs c\rs  =& -\sumrs[M][N] g\rs (\delr \Jr\rs + \dels \Js\rs)\\
    \overset{\text{summation by parts}}{=}&
    \sumrs[M][N] \delr g\rps \Jr\rs + \dels g\rsp \Js\rs\\
    &+\sumrs[M-1][] g\rps \Jr\rs + g\rsp \Js\rs
    -\sumrs[N][] g\rps \Jr\rs + g\rsp \Js\rs.
  \end{split}
\end{equation}
Taking $M=2$ and formally letting $N\to \infty$, we see that the equations for the monomers are chosen, so that 
\begin{equation}\label{mcBDWeak}
    \sumrs g\rs c\rs(t)  = \sumrs g\rs c\rs(0) + \int_0^t \sumrs (\delr g\rps - g_{1,0}) \Jr\rs + (\dels g\rsp - g_{0,1}) \Js\rs\dd \tau.
\end{equation}
Taking $g\rs = r $ and $g\rs = s$, we formally find $\rho(t) = \rho(0)$ and $\sigma(t) = \sigma(0).$

Even though the system \eqref{mcBD} (usually referred to as birth--death equations) has been used extensively to calculate/estimate nucleation rates for atmospheric droplets \cite{elm:applicationReview,vehkamaeki:CNT}, it has received very little attention in the mathematical literature. We are only aware of a brief and early study as part of dissertation \cite{dunwell:phd}. The goal of the present work, is to establish a sound foundation for the study of \eqref{mcBD}. Hence, we do not want to limit ourselves to some specific coefficients, but establish as general results as possible. For the general existence results, we only need $\ar\rs+\as\rs \in \bigO(r+s)$ (which is know to be ``optimal'' in the one-component system \cite{penrose:foundations}). To study the entropy and equilibria, we need to assume detailed balance and an analogue of $0<z_s<\infty$ given in \eqref{assCoeff}. 

\subsection{Detailed Balance}

Unless energy or entropy goes into our out of the system, the coefficients $\ar\rs,\as\rs,\br\rs,\bs\rs$ should satisfy the detailed balance assumption.
\begin{definition}(Detailed balance)\label{defiDetailedBalance}\\
  We say that $0<\ar\rs,\as\rs,\br\rps,\bs\rsp$ satisfy detailed balance, if there exists a sequence $\cbar\rs> 0$, such that 
  \begin{equation}\label{equilibirum}
    \Jr\rs(\cbar) \equiv 0 \equiv \Js\rs(\cbar) \quad\text{ for all }(r,s)\in\Omega.
  \end{equation}
  Such a $\cbar\rs$ is called equilibrium.
\end{definition}
This is equivalent (Proposition \ref{detailedBalance}) to the existence of a sequence $0<Q\rs$, with $Q_{1,0} = Q_{0,1} =1$, so that 
\begin{equation}\label{bFromQ}
  \br\rps = \ar\rs \frac{Q\rs}{Q\rps} \quad\text{ and }\quad\bs\rsp = \as\rs \frac{Q\rs}{Q\rsp}.
\end{equation}
Note, that \eqref{bFromQ} really is an assumption and not always given, as opposed to the one-component system, where detailed balance always holds via \eqref{ocDetailedBalance}. However, under this assumption we can again find an entropy. Formally we have an entropy $V(c)$ for solutions $c(t)$ of \eqref{mcBD} due to 
\begin{equation}
  \odv{}{t} \underbrace{\sumrs c\rs \left(\ln\left(\frac{c\rs}{Q\rs}\right) - 1\right)}_{\eqqcolon V(c)}
  = - \sumrs \Jr\rs \ln\left(\frac{\ar\rs\monr c\rs}{\br\rps c\rps}\right) + \Js\rs\ln\left(\frac{\as\rs\mons c\rs}{\bs\rsp c\rsp}\right)\le 0,
\end{equation}
which we can establish rigorously under the assumption 
\begin{equation}\label{assCoeff}
   \begin{split}
       \lim\limits_{N\to\infty}\inf\limits_{r+s\ge N} Q\rs^\frac 1 {r+s} >0 \quad\text{ and }\quad
       \lim\limits_{N\to\infty}\sup\limits_{r+s\ge N} Q\rs^\frac 1 {r+s} <\infty,
   \end{split}
\end{equation}
together with a higher moment bound for the initial conditions and a growth condition on the coagulation and fragmentation parameters.
This naturally leads us to the study of all equilibria with finite masses. While this is much more intricate in the two-component system, we can still fully characterise them. Since the goal of this work is to establish a general foundation for \eqref{mcBD}, we do not want to go into detail on how to derive $Q\rs$ for a physical system. But let us mention a very standard form $Q\rs = \lambda^r \mu^s \binom{r+s}{r}e^{-(r+s)^{\frac 2 3}}$ for some $0<\lambda,\mu<\infty,$ used to describe liquid droplet formation \cite{wilemski:Coefficients}. These do satisfy \eqref{assCoeff}, which is a very tame assumption and true for many physically relevant systems \cite{vehkamaeki:CNT}.

\subsection{Outline}

We start our discussion of \eqref{mcBD} in Subsection \ref{existenceMassConservation} with general existence, mass conservation and uniqueness results. The main tool is an a priori estimate (Lemma \ref{mcBDApriori}) based on the de la Vallée-Poussin description for uniform integrability. This follows the ideas of \cite{laurencot:existence}. These results may be summed up as 
\begin{theorem}(Cauchy problem \eqref{mcBD})\\
  Assume that $\ar\rs + \as\rs \in \bigO(r+s)$. Then, there exists a solution to \eqref{mcBD} for all times. Furthermore, any solution will conserve Type I and Type II mass for all times. If additionally $\sumrs (r+s)^2 c\rs(0) < \infty,$ the solution is unique.
\end{theorem}
The existence and mass conservation parts are complete generalisations of the one-component system. However, the uniqueness result is less sharp. This is effectively caused by the problem of establishing the weak formulation for general test sequences.

After a characterisation of the detailed balance assumption (Proposition \ref{detailedBalance}), we rigorously prove the entropy equation  in Subsection \ref{sectionEntropyEquation}. This naturally warrants a detailed study of equilibria and relative entropies. \\
We start this in Section \ref{sectionSteadyStates}, where we describe equilibria with finite mass as limits of finite steady states, defined on some finite subset $\Gamma\subset\Omega$. We will demonstrate, that all equilibria can be determined by a two parameter family  over some two-dimensional manifold $\Ex\subset \R_{\ge0}^2$ and are given as $\cinfzw\rs \coloneqq z^r w^s Q\rs$, for all $(z,w) \in \Ex$. The entropy structure between steady states will allow us to find a nice description of all equilibria, which states that for any masses $\rho,\sigma >0$, there exists a unique $(z,w) \in \Ex$, such that 
\begin{equation}
  \ln\left(\frac z u\right)\left(\rho_0 - \rho(\cinfzw)\right) + \ln\left(\frac w v\right) \left(\sigma_0 - \sigma(\cinfzw)\right) \ge 0 \text{ for all } 0< u,v \text{ with }(u,v) \in \Ex.
\end{equation}
As in the one-component system, starting from the entropy $V(c)$, we can naturally find entropies relative to any equilibrium $\cinfzw\rs$, which we will denote $H[c|\cinfzw]$ for any $(z,w)\in \Ex$. In the one-component system, it is know that solutions to \eqref{ocBD} minimise the relative entropy \cite{penrose:foundations}. Therefore, it is natural to study the minimisation of the relative entropies under the constraint of fixed Type I  and Type II mass. We can determine the minimum (Lemma \ref{relativeEntropyTwoSteadyStates} and Theorem \ref{relativeEntropyMinimum}) and describe the minimising sequences in the weak* topology (Corollary \ref{relativeEntropyMinimizingSequences}). Finally we will extend these results to the quasi steady state.

\section{Basic Properties}\label{basicProperties}
\subsection{Existence, Mass Conservation and Uniqueness}\label{existenceMassConservation}
We begin our study of \eqref{mcBD} with the general Cauchy problem. Existence, conservation of mass and positivity of solutions are mostly a straight forward extension of the one-component system. As is often the case with coagulation equations, proving uniqueness turns out to be a delicate problem. In particular, the strong results from the one-component system do not translate as easily.\\
Let us start with the functional analytic setting of our solutions spaces. Remember, that $\Omega \coloneqq \{(r,s) \in \N^2_0 \,|\, r+s\ge 1\}$ and all the sums, have to be understood as ranging over subsets of $\Omega.$
    \begin{definition}\textit{(Solution space)}\label{solutionSpace}\\
        We define the Banach space
\begin{equation*}
X\coloneqq \left\{(c\rs)_{r,s\in\Omega} \subset \R \;\middle|\; \sumrs r |c\rs| < \infty \text{ and } \sumrs s |c\rs| < \infty \right\},
\end{equation*}
together with its norm
\begin{equation*}
    \norm{c}_X\coloneqq \sumrs (r+s) |c\rs|.
\end{equation*}
Furthermore, we define the positive cone of $X$ as
\[X^+\coloneqq \left\{c\in X \mid c\rs \ge0 \text{ for all } r,s\in\Omega \right\}.\]

\end{definition}
    \begin{proposition}\textit{(Weak* topology)}\label{weakTopology}\\
        Let $(c^n)\subset X$ and $c\in X$. Then, $c^n \weakstar c$, iff
\begin{equation}\label{proposition:weakTopology}
  \sup\limits_n \norm{c^n}<\infty \text{ and } c\rs^n \to c\rs \text{ as }n\to\infty \text{ for all }r+s\ge 1.
\end{equation}

\end{proposition}
\begin{proof}
    Note that $X = Y^*$, with 
    \[Y \coloneqq \left\{y\rs \middle| \frac{y\rs}{r+s} \to 0 \text{ as } r+s \to \infty\right\},\]
    and norm $\norm{y}_Y = \sup \frac {|y\rs|}{r+s}$. This identification can easily be seen, once we notice, that sequences with only finitely many non zeros are dense in $Y$. Now one quickly verifies that \eqref{proposition:weakTopology} holds true, iff 
    \[\sumrs c\rs^n y\rs \to \sumrs c\rs y\rs \text{ for all }y\in Y.\]
\end{proof}
As in \cite{penrose:foundations}, this clearly yields a metric.
    \begin{proposition}\textit{(Weak* metric)}\label{weakMetric}\\
        Let $\rho,\sigma \ge0$ and define $\Brohsigp \coloneqq \left\{c \in X^+ \middle| \sumrs r c\rs \le \rho, \sumrs s c\rs \le \sigma\right\}$. Then, $\Brohsigp$ equipped with the metric $d(x,y) = \sumrs |x\rs - y\rs|$ is compact and its induced topology coincides with the weak* topology, i.e. for $c^n,c \in \Brohsigp$, we have
\[ c^n \weakstar c \iff d(c^n,c)\to 0.\]

\end{proposition}
With this, we define solutions in the usual manner, as mild solutions.
    \begin{definition}\textit{(Solutions to \eqref{mcBD})}\label{mcBDSolutions}\\
        Fix $0<T\le\infty$. We call a function $c \colon [0,T) \to X^+$ a solution to \eqref{mcBD} with initial data $c^0 \in X^+$, if 
\begin{enumerate}
    \item for all $(r,s)\in\Omega$ the function $c\rs \colon [0,T)\to\R_{\ge0}$ is continuous and  $\sup\limits_{[0,T)}\norm{c(t)}<\infty$,
    \item the transitions are in $L^1_\loc$, i.e.
        \begin{equation*}
            \int_0^t \sumrs (\ar\rs+\as\rs + \br\rps + \bs\rsp) c\rs(\tau) \dd\tau < \infty \text{ for all } t\in[0,T)\text{ and }
        \end{equation*}
    \item the equations are satisfied in a mild sense, i.e. for any $t\in[0,T)$
        \begin{equation}
            \begin{split}
                c_{1,0}(t) &= c^0_{1,0} - \int_0^t \Jr_{1,0}(c(\tau)) + \Js_{1,0}(c(\tau)) + \sumrs \Jr\rs(c(\tau)) \dd\tau,\\
                c_{0,1}(t) &= c^0_{0,1} - \int_0^t \Jr_{0,1}(c(\tau)) + \Js_{0,1}(c(\tau)) + \sumrs \Js\rs(c(\tau)) \dd \tau \text{ and}\\
                c\rs(t) &= c^0\rs - \int_0^t \delr \Jr\rs(c(\tau)) + \dels \Js\rs(c(\tau))\dd\tau \text{ for all }r+s\ge2.
            \end{split}
        \end{equation}
\end{enumerate}

\end{definition}
Our proofs of existence, mass conservation and uniqueness heavily rely on the de la Vallée-Poussin description of uniform integrability---similarly to \cite{laurencot:existence}. Let us state the exact results we use for the readers convenience.
    \begin{theorem}\textit{(de la Vallée-Poussin \cite{delaVallee})}\label{delaVallee}\\
        Let $(X,\Sigma,\mu)$ be a measure space and $B\subset L_1(X)$ be bounded. Then, $B$ is uniformly integrable, if and only if there is a convex function $\varphi\in C^\infty([0,\infty))$, such that 
\[ \sup\limits_{f\in B}\int_X \varphi(|f|) \dd \mu < \infty,\]
and $\varphi$ satisfies the following properties:
$\varphi(0) = 0 =\varphi^\prime(0)$, $\varphi^\prime$ is concave and strictly positive on $(0,\infty)$,
\[ \lim\limits_{r \to \infty} \varphi^\prime(r) = \lim\limits_{r \to \infty} \frac{\varphi(r)} r = \infty\]
and 
\[ \lim\limits_{r \to \infty} \frac{\varphi^\prime(r)}{r^{p-1}} = 0 \text{ for all } p\in(1,2].\]

    \end{theorem}
We will also use the following standard result for these functions $\varphi$.
    \begin{lemma}\textit{}\label{helpDelaVallee}\\
        Let $\varphi\colon [0,\infty) \to [0,\infty)$ satisfy the properties from Theorem \ref{delaVallee}. Then, for any $x\ge0$ and $y\ge0$ 
\[(x+y) [\varphi(x+y) - \varphi(x) -\varphi(y)] \le 2 [y\varphi(x) + x \varphi(y) ]\]
holds true.

    \end{lemma}
\begin{proof}
  A proof for Theorem \ref{delaVallee} and Lemma \ref{helpDelaVallee} can be found in \cite{coagulation:book}, where they are stated as Theorem 7.1.6. and Proposition 7.1.9. respectively.
\end{proof}
The heart of our existence, mass conservation and uniqueness results is the following a priori estimate.
    \begin{lemma}\textit{(A priori estimate)}\label{mcBDApriori}\\
        Assume that $\ar\rs + \as\rs \in \bigO(r+s)$ and that for some $\beta \ge 0$ and $\varphi$ as in Theorem \ref{delaVallee}, we have 
\[ \sumrs \varphi(r+s) (r+s)^\beta c^0\rs < \infty.\]
Then, there exists a constant $K<\infty$, such that for any solution of \eqref{mcBD} on $[0,T)$ with initial data $c^0$, we have 
\[ \sumrs \varphi(r+s) (r+s)^\beta c\rs(t) < K \text{ for all }0\le t < T.\]
This constant $K$ depends on a bound of 
\[ \max \left( \varphi(1),\varphi(2),\sup\limits_{x\ge2}\frac{x}{\varphi(x)}, \beta,T, \sup\limits_{t\in[0,T)}||c(t)||, \sup\Big\{\frac{\ar\rs + \as\rs}{r+s}\Big\},\sumrs \varphi(r+s) (r+s)^\beta c^0\rs\right).\]

    \end{lemma}
\begin{proof}
  In this proof, the constant $K$ may change in every line without renaming, but always satisfies the promised dependencies and is in particular independent of $N$. Also we will denote for a one-dimensional sequence $g_m$ the discrete backwards derivative as $\del g_m \coloneqq g_m - g_{m-1}.$\\
For $t\in(0,T)$ and $N\in \N$ arbitrary,  Definition \ref{mcBDSolutions} (iii) implies 
\begin{equation}\label{mcBDApriori:eq1}
    \begin{split}
      \sumrs[2][N] & \vphi(r+s)(r+s)^\beta c\rs(t)
      \overset{\text{def sol.}}{=} \sumrs[2][N] \vphi(r+s) c^0\rs - \int_0^t \sumrs[2][N]\vphi(r+s)(r+s)^\beta (\delr \Jr\rs + \dels \Js\rs)\dd \tau\\
      \overset{\text{summation by parts}}&{=} \int_0^t \sumrs[2][N] \del\Big( \vphi(r+s+1)(r+s+1)^\beta\Big) (\Jr\rs+\Js\rs)\dd\tau +\int_0^t \sumrs[1][]2^\beta \vphi(2)(\Jr\rs+\Js\rs)\dd \tau\\
      &\qquad\quad -\int_0^t \sumrs[N][] \vphi(N+1)(N+1)^\beta (\Jr\rs+\Js\rs)\dd \tau.
    \end{split}
\end{equation}
Now, we note that Lemma \ref{helpDelaVallee} with $x=r+s$ and $y=1$ implies for $r+s \ge 2$
\begin{equation}\label{mcBDApriori:eq2}
  \begin{split}
    \del\Big(& \vphi(r+s+1)(r+s+1)^\beta\Big) \\
    &=  (r+s+1)^{\beta-1}\del \vphi(r+s+1) + \varphi(r+s) \del (r+s+1)^\beta\\
    \overset{\text{Lemma \ref{helpDelaVallee}}}&= \underbrace{(r+s+1)^\beta}_{\le K (r+s)^{\beta-1}} \Big(\underbrace{(r+s+1)}_{\le K \varphi(r+s)} \varphi(1)+ 2[\varphi(r+s) + \underbrace{(r+s)}_{\le K \varphi(r+s)} \varphi(1)]\Big) + \varphi(r+s) \del (r+s+1)^\beta\\
    \overset{\del (r+s+1)^\beta \le K (r+s)^{\beta-1}}&\le K (r+s)^{\beta-1} \varphi(r+s).
  \end{split}
\end{equation}
Since $x^\beta\varphi(x)$ is a monotonically increasing function, we have $\del\Big(\vphi(r+s+1)(r+s+1)^\beta\Big)\ge 0$ and hence 
\begin{equation}
  \begin{split}
    \int_0^t &\sumrs[2][N] \del\Big( \vphi(r+s+1)(r+s+1)^\beta\Big) (\Jr\rs+\Js\rs)\dd\tau\\
    &\le \int_0^t \sumrs[2][N] \del\Big( \vphi(r+s+1)(r+s+1)^\beta\Big) (\ar\rs\monr c\rs +\as\rs\mons c\rs)\dd\tau\\
    \overset{\text{\eqref{mcBDApriori:eq2}}}&\le K \int_0^t \sumrs[2][N] \vphi(r+s+1)(r+s+1)^\beta c\rs \dd\tau.
  \end{split}
\end{equation}
The lower boundary term can be estimated by $\int_0^t \sumrs[1][]2^\beta \vphi(2)(\Jr\rs+\Js\rs)\dd \tau\le K.$ And lastly, we find
\begin{equation}
  \begin{split}
    -\int_0^t & \sumrs[N][] \vphi(N+1)(N+1)^\beta (\Jr\rs+\Js\rs)\dd \tau \\
    \overset{\text{Def. \ref{mcBDSolutions} (ii)}}&= \vphi(N+1)(N+1)^\beta \int_0^t \sumrs[M+1] (\delr \Jr\rs+\ \dels \Js\rs)\dd \tau \\
    \overset{\text{Def. \ref{mcBDSolutions} (iii)}}&= \vphi(N+1)(N+1)^\beta \sumrs[M+1] c\rs(0) - c\rs(t) \\
    \overset{\varphi(x) x^\beta \text{monotone}}&\le \sumrs[M+1] \varphi(r+s) (r+s)^\beta c\rs(0).
  \end{split}
\end{equation}
Putting everything into \eqref{mcBDApriori:eq1} yields
\begin{equation}
      \sumrs[2][N] \vphi(r+s)(r+s)^\beta c\rs(t) \le K\left(1+\int_0^t \sumrs[2][N] \vphi(r+s) (r+s)^\beta c\rs\dd\tau\right),
\end{equation}
so we can conclude with Grönwall.

\end{proof}
Now, we can prove that solutions exist, given that the coagulation coefficients do not grow too quickly. This is completely analogous to the one-component system. Also, the same result was already shown in \cite{dunwell:phd} generalising the proof of \cite{penrose:foundations}. We still want to show our proof, based on Theorem \ref{delaVallee} (which is considerably shorter), so that the reader can understand the subtle difference between the one-component and two-component system that we will encounter in the uniqueness proof.
    \begin{theorem}\textit{(Existence to \eqref{mcBD})}\label{mcBDExistenceDela}\\
        Let $c^0\in X^+$ and assume that $\ar\rs + \as\rs \in \mathcal{O}(r+s).$
Then, there exists a solution $c$ to \eqref{mcBD} with initial data $c^0$ on $[0,\infty)$.

    \end{theorem}
\begin{proof}
  In this proof, $K$ denotes a constant independent of $n,r,s$ and $C(r,s)$ or $C(r+s)$ a constant independent of $n$. Both may change from line to line without renaming.\\
We cut the equation off at $r+s=n$ via $\Jr\rs = \Js\rs = 0$ for $r+s\ge n$, i.e. we set $\ar^n\rs = \as^n\rs = \br^n\rps = \bs^n\rsp = 0$ for all $r+s\ge n$ and $c^{0n} = 0$ for all $r+s >n$ and get solutions $c^n\ge 0$ by classical ODE theory. We have conservation of mass and in particular $c^n\rs \le \norm{c^0}$ and for $r+s\ge2$ also $|\odv{}{t}c^n\rs| \le C(r,s) \norm{c^0}^2$. So up to some diagonal sequences we have
\begin{equation*}
    \begin{split}
        &c^n\rs \to c\rs \quad \text{ uniformly on compact intervals for }r+s\ge2\\
        &c^n_{1,0}\weakstar c_{1,0} \quad\text{ in }L^\infty\\
        &c^n_{0,1}\weakstar c_{0,1} \quad\text{ in }L^\infty.
    \end{split}
\end{equation*}
Let us start by showing part (ii) from the definition, i.e. that the transitions are in $L^1_\loc$. To this end we use $g\rs = r+s$ in \eqref{weakmcBDfinite} and take $M=2$ and $N=\infty$. From this we get 
\begin{equation*}
    \begin{split}
        \odv{}{t}\sumrs[2](r+s)c^n\rs + \sumrs[2] \br\rps c^n\rps + \bs\rsp c^n\rsp
    &= \sumrs[2] (\ar\rs c^n_{1,0}+\as\rs c^n_{0,1})c^n\rs + 2\sumrs[1][] \Jr\rs+\Js\rs\\
    &\le K(1+\sumrs[2](r+s)c^n\rs).
    \end{split}
\end{equation*}
Now Grönwall and integrating this up yields the desired bounds uniformly in $n$.
Next, we want to upgrade the convergence of the monomers to uniform convergence. To that end we apply Theorem \ref{delaVallee} with the measure space $(\N,\mathcal P(\N), \sumrs[m][]c^0\rs)$ and the single sequence $B = \{(m)_{m\in \N}\}$ to find a $\varphi$ as in Theorem \ref{delaVallee} such that $\summ \vphi(m)\sumrs[m][]c^0\rs<\infty$. By Lemma \ref{mcBDApriori} with $\beta = 0$, $\sumrs[2]\vphi(r+s)c^n\rs$ is bounded on compact intervals uniformly in $n$. Now, we cannot bound the derivatives of the monomers easily, because we only have $\sumrs \br\rps c^n\rps \in L^1$ and not in $L^\infty$. So we try to bound the integrated version. For this purpose, we first start with the following, again using $g\rs= r+s$ in \eqref{weakmcBDfinite} and integrating from $0$ to $\tau$, to obtain for any $M$
\begin{equation*}
    \begin{split}
        \sumrs[M] &(r+s)c^n\rs(\tau) + \int_0^\tau \sumrs[M] \br\rps c^n\rps + \bs\rsp c^n\rsp \dd \sigma\\
        &= \sumrs[M] (r+s) c^n\rs(0) + \int_0^\tau \sumrs[M]\underbrace{(\ar\rs c^n_{1,0} + \as\rs c^n_{0,1})}_{\le K(r+s)}c^n\rs
        + \underbrace{\int_0^\tau \sumrs[M-1][] M (\Jr^n\rs + \Js^n\rs)}_{= M (\sumrs[M]c^n\rs(\tau)-\sumrs[M] c^n\rs(0))}.
    \end{split}
\end{equation*}
Next, we integrate this from $0$ to $t$ 
\begin{equation*}
    \begin{split}
        \int_0^t &\sumrs[M](r+s)c^n\rs(\tau) +\int_0^t\int_0^\tau \sumrs[M]\br\rps c^n\rps + \bs\rsp c^n\rsp \dd \sigma \dd \tau\\
                 &\le \underbrace{\int_0^t \sumrs[M](r+s)c^n\rs(0)\dd\tau}_{\to 0} + \underbrace{\int_0^t\frac{M}{\vphi(M)} \sumrs[M]\vphi(r+s)c^n\rs \dd \tau}_{\to 0} + K \int_0^t\int_0^\tau \sumrs[M](r+s)c^n\rs\dd \sigma \dd\tau,
    \end{split}
\end{equation*}
where Grönwall tells us, that the first line goes to zero as $M\to\infty$ uniformly in $n$. With this we can show that $\int_0^t \sumrs[M] |\Jr^n\rs| + |\Js^n\rs| \dd\tau \to 0$ as $M\to\infty$ uniformly in $n$. Fix $T>2$ and let $t_1<t_2\in[0,\frac{T}{2}]$, then we have for any $M$ 
\begin{equation}\label{mcBDExistenceProof:eq1}
    \begin{split}
      \int_{t_1}^{t_2}&\sumrs[M] |\Jr^n\rs|+|\Js^n\rs| \dd\tau\\
       & \le \int_0^T \sumrs[M] K(r+s)c^n\rs \dd \tau + \int_0^{\frac{T}{2}} \underbrace{\frac{T}{2}}_{\le T-\tau}\sumrs[M] (\br\rps c^n\rps + \bs\rsp c^n\rsp) \dd \tau\\
       &\le\int_0^T \sumrs[M] K(r+s)c^n\rs \dd \tau + \underbrace{\int_0^{T} (T-\tau)\sumrs[M](\br\rps c^n\rps + \bs\rsp c^n\rsp) \dd \tau}_{=\int_0^T\int_0^\tau \sumrs[M](\br\rps c^n\rps + \bs\rsp c^n\rsp)\dd \sigma \dd \tau},
    \end{split}
\end{equation}
which can be made arbitrarily small by taking $M$ large uniformly in $t_1,t_2$ and $n$ according to the above estimate. This yields uniform continuity of the monomer densities via
\begin{equation}
    \begin{split}
        |c^n_{1,0}(t_1)&-c^n_{1,0}(t_2)|+|c^n_{0,1}(t_1)-c^n_{0,1}(t_2)|\\
        &= \left|\int_{t_1}^{t_2} \Jr^n_{1,0}+\Js^n_{1,0} + \sumrs \Jr^n\rs \dd \tau\right|
        +\left|\int_{t_1}^{t_2} \Jr^n_{0,1}+\Js^n_{0,1} + \sumrs \Js^n\rs \dd \tau\right|\\
        &\le C(M) |t_1-t_2| + \int_{t_1}^{t_2} \sumrs[M] |\Jr^n\rs|+|\Js^n\rs| \dd\tau,
    \end{split}
\end{equation}
 so we can apply Arzelà--Ascoli, i.e.
\[c^n_{1,0} \to c^{1,0} \text{ and } c^n_{0,1} \to c_{0,1} \text{ uniformly on compact intervals.}\]
Finally, since $\int_0^t \sumrs[M] |\Jr^n\rs| + |\Js^n\rs| \dd\tau \to 0$ as $M\to\infty$ uniformly in $n$, we can pass to the limit in the equations for the monomers.

\end{proof}
From here, we can quickly show that any solution conserves both Type I and Type II mass. Nevertheless, the argument is somewhat more involved than the extension of the existence proof. In the one-component system, when handling the boundary term $\int_0^t g_{N+1} J_N\dd \tau$ (for some test-sequence $g_r$), the trick is always to go to $g_{N+1} \sum\limits_{n=N+1}^{\infty} c_n(t) - c_n(0)$. In the two-component system, we have on the boundary (denoted $\Gamma$) of some ``nice'' finite set $\sum\limits_{r,s \in \Gamma} g\rps \Jr\rs + g\rsp \Js\rs$. So we have to find a $\Gamma$, such that $g\rps \equiv \const \equiv g\rsp$ on $\Gamma.$ This is not exactly possible for $g\rs = r$ or $g\rs = s$, but we can effectively test on $\{r\le R\}$ for some large $R$, which is then appropriate.
    \begin{theorem}\textit{(Mass conservation)}\label{massConservation}\\
        Let $c^0 \in X^+$ and $\ar\rs + \as\rs \in \mathcal{O}(r+s)$. Then, any solution to \eqref{mcBD} with initial data $c^0$ on $[0,T)$ satisfies
\begin{equation}
    \sumrs r c\rs (t) = \sumrs r c^0\rs \text{ and } \sumrs s c\rs(t) = \sumrs c^0\rs \text{ for any } t<T.
\end{equation}

    \end{theorem}
\begin{proof}
    Let $\vphi(m)$ again be a de la Vallée function such that $\sumrs \vphi(r+s) c^0\rs<\infty$. Furthermore let $c$ be a solution, $t<T$ and $2<R,S\in \N$ be arbitrary. Then, we may calculate
\begin{equation}\label{helperEquationMassConservation}
   \begin{split}
       \sumrs[2] r(c\rs(t)-c^0\rs) &\gets \sumrs[2][\substack{r=R \\s=S}] r (c\rs(t)-c^0\rs)\\
       \overset{\text{def}}&{=} \underbrace{\int_0^t\sumrs[2][\substack{r=R \\s=S}] \Jr\rs+  \sumrs[1][] (r+1)\Jr\rs + r \Js\rs \dd \tau}_{\to c_{1,0}(t)-c_{1,0}(0) \text{ by definition}}\\
       &\quad-\int_0^t\sum\limits_{s=0}^S (R+1) \Jr_{R,s} \dd \tau 
       - \int_0^t \sum\limits_{r=0}^R r \Js_{r,S} \dd \tau.
   \end{split} 
\end{equation}
So if we show that the boundary term vanishes as $S,R\to\infty$, we are done. To that end we calculate
\begin{equation*}
    \int_0^t \sum\limits_{s=0}^S (R+1)\Jr_{R,s}\dd \tau \overset{\text{dom. conv}}{\to}
    (R+1)\int_0^t \sum\limits_{s=0}^\infty \Jr_{R,s}\dd \tau 
= (R+1) \sum\limits_{r=R+1}^\infty \sum\limits_{s=0}^\infty c\rs(t)-c\rs(0) \to 0,
\end{equation*}
where the last convergence is due to
\[
  (R+1)\sum\limits_{r=R+1}^\infty \sum\limits_{s=0}^\infty c\rs(t)\le \underbrace{\frac{(R+1)}{\vphi(R+1)}}_{\to 0}\underbrace{\sum\limits_{r=R+1}^\infty \sum\limits_{s=0}^\infty \vphi(r+s)c\rs(t)}_{\le K \text{ by Lemma \ref{mcBDApriori}}}.
\]
So the $r=R$ term vanishes, if we take $R$ large and then $S$ large. Lastly for a fixed $R$ large, the $s=S$ term goes to zero by dominated convergence as $S\to\infty$.\\
The Type II conservation is proved analogously.

\end{proof}
A simple consequence of Theorem \ref{massConservation} is the continuity in $X$.
    \begin{proposition}\textit{(Continuity in $X$)}\label{solutionContinuity}\\
        Let $\ar\rs,\as\rs \in \mathcal{O} (r+s)$. Then, any solution to \eqref{mcBD} is already continuous as a function $c\colon [0,T)\to X^+$. Furthermore, the sum $\summ m \sumrs[m][] c\rs$ is uniformly convergent on compact intervals.

\end{proposition}
\begin{proof}
  We can argue as in \cite{penrose:foundations}.
    The sequence $\summ[1][N]m\sumrs[m][]c\rs$ is a monotone sequence of continuous functions, which converges by Theorem \ref{massConservation} to a continuous function. Dini's theorem shows, that the convergence is therefore uniform on compact intervals. From this we can conclude the continuity in $X$.

\end{proof}%
Lastly, we want to prove uniqueness of solutions. Unfortunately, the nice uniqueness proof of Laurençot \cite{laurencot:uniqueness} does not seem to generalise so easily. One may try to use the tail distributions $F\rs \coloneqq \sum\limits_{\substack{k\ge r \\ l\ge s}} c\kl$, which do satisfy $\delr\dels F\rpsp = c\rs,$ but we then run into trouble, when summing up their time derivatives. Furthermore, it is not clear if the trick of multiplying $\odv{}{t} (F\rs - \tilde F\rs)$ with $\sgn(F\rs - \tilde F\rs)$ will work.
  Instead, we use a direct approach as in \cite{penrose:foundations} combined with de la Vallée to estimate $\odv{}{t} ||c - \tilde c||$ for two solutions $c$ and $\tilde c$, which yields uniqueness of solutions, whenever the initial conditions have finite second moment.
    \begin{theorem}\textit{(Uniqueness for bounded second moments)}\label{mcBDUniqueness}\\
        Let $c^0 \in X^+$, such that $\sumrs (r+s)^2 c^0\rs <\infty$ and $\ar\rs + \as\rs \in \mathcal{O}(r+s).$
Then, there exists exactly one solution of \eqref{mcBD} on any interval $[0,T)$.

    \end{theorem}
\begin{proof}
    In this proof, $K$ denotes a constant independent of $c,\tilde c$ and $t$ that may change from line to line without renaming.\\
Let $c\rs$ and $\tilde{c}\rs$ be two solutions, which have uniformly bounded second moment by Lemma \ref{mcBDApriori}. We denote with $x\rs = c\rs - \tilde{c}\rs$, $\sgn\rs = \sgn(x\rs)$, $\Ir\rs = \Jr\rs(c) - \Jr\rs(\tilde{c})$ and $\Is\rs = \Js\rs(c) - \Js\rs(\tilde{c})$. With this we obtain for any $N$ and $t<T$
\begin{equation}\label{helperEquationUniquness}
    \begin{split}
        \sumrs[2][N] (r+s)|x\rs|(t)
        &=\int_0^t -\sumrs[2][N] \sgn\rs (r+s) (\delr \Ir\rs - \dels \Is\rs)\dd \tau\\
        \overset{\text{summation by parts}}&{=}\int_0^t \sumrs[2][N] \delr(\sgn\rps(r+s+1))\Ir\rs+\dels(\sgn\rsp(r+s+1))\Is\rs \dd \tau\\
                               &\qquad+\int_0^t \sumrs[1][]2\sgn\rps \Ir\rs + 2 \sgn\rsp \Is\rs \dd \tau\\
                               &\qquad- \int_0^t\sumrs[N][]\sgn\rps (r+s+1)\Ir\rs + \sgn\rsp(r+s+1)\Is\rs \dd \tau\\
                               &\eqqcolon \int_0^t A + B + C \dd \tau.
    \end{split}
\end{equation}
Now we note that we can write 
\[\Ir\rs = \ar\rs c\rs x_{1,0} + \ar\rs \tilde{c}_{1,0}x\rs - \br\rps x\rps,\]
to find for the first term of $A$
\begin{equation*}
    \begin{split}
        \delr(\sgn\rps (r+s+1))\Ir\rs &= 
        \overbrace{\delr(\sgn\rps (r+s+1))}^{|\cdot|\le 2r+2s+1}\ar\rs c\rs x_{1,0}\\
                                      &\qquad+ \sgn\rps(r+s+1)\ar\rs \tilde{c}_{1,0} x\rs - (r+s)\ar\rs \tilde{c}_{1,0} |x\rs|\\
        &\qquad+ \br\rps x\rps \sgn\rs(r+s)  - \br\rps |x\rps|(r+s+1)\\
        &\le K (r+s) \ar\rs c\rs |x_{1,0}| + \ar\rs \tilde{c}_{1,0}|x\rs|.
    \end{split}
\end{equation*}
We can treat the term with $\dels$ similarly and sum over $r+s=2$ to $N$ to obtain
\begin{equation}
    A \le K \sumrs[2][N] (r+s)^2 c\rs |x_{1,0}| + \tilde c_{1,0} (r+s)|x\rs| + (r+s)^2 c\rs|x_{0,1}| + \tilde c_{0,1} (r+s)|x\rs|.
\end{equation}
Next, we treat the boundary term at $r+s=1$, i.e. $B$. For this we simply estimate
\begin{equation*}
    \sumrs[1][] \sgn\rps \Ir\rs + \sgn\rsp \Is\rs
    \le \sumrs[1][] \ar\rs c\rs |x_{1,0}| + \ar\rs \tilde c_{1,0} |x\rs|+ \as\rs c\rs |x_{0,1}| + \as\rs \tilde c_{0,1} |x\rs|.
\end{equation*}
Since $|x\rs|$ is just $|x_{1,0}|$ or $|x_{0,1}|$ in the sum over $r+s=1$, we can combine this estimate together with the one on $A$ and the bound on $\sumrs (r+s)^2 c\rs$ to obtain
\begin{equation}
    A+B \le K\left(|x_{1,0}| + |x_{0,1}| + \sumrs[2][N] (r+s)|x\rs|\right).
\end{equation}
But now we can treat the difference in the monomers via the conservation of mass
\[|x_{1,0}| = |c_{1,0}-\tilde{c}_{1,0}| = | \sumrs[2] r c\rs - \sumrs[2]r \tilde{c}\rs| \le \sumrs[2](r+s)|x\rs|.\]
Plugging this estimate into \eqref{helperEquationUniquness} yields
\begin{equation}
    \sumrs[2][N] (r+s)|x\rs|(t) \le K  \int_0^t \sumrs[2](r+s)|x\rs| \dd \tau 
    + \int_0^t C_N(\tau)\dd \tau.
\end{equation}
Therefore, it suffices to show, that the term involving $C_N$ goes to zero as $N\to\infty$, which would allow us to conclude via Grönwall. To do this we follow the same ideas as in the existence proof. We once more take a de la Vallée function $\vphi$ depending on $c^0$, such that $\sumrs[1][\infty][m] \vphi(m) \sumrs[m][] (r+s)c^0\rs<\infty$, which is possible due to the second moment bound. From here we proceed as follows
\begin{equation}\label{macBDUniquness:eq1}
    \begin{split}
        \int_0^t& \sumrs[N][]\sgn\rps (r+s+1)\Ir\rs + \sgn\rsp(r+s+1)\Is\rs \dd\tau\\
        &\le \int_0^t (N+1)\sumrs[N]|\Jr\rs(c)|+|\Js\rs(c)| + |\Jr\rs(\tilde{c})| + |\Js\rs(\tilde{c})| \dd \tau.
    \end{split}
  \end{equation}
We will show, that $N \Big(\sumrs[N](r+s)c\rs(\tau) + \int_0^\tau \sumrs[N] \br\rps c\rps + \bs\rsp c\rsp \dd \sigma\Big)$ goes to zero as $N\to\infty$. Then, recovering the integrated version is done as in \eqref{mcBDExistenceProof:eq1} of the existence proof. Theorem \ref{massConservation} implies that $g\rs = \indicator{r+s\ge N} (r+s) = (r+s) - \indicator{r+s< N } (r+s)$ is a valid test sequence for \eqref{mcBDWeak} yielding for any $0<\tau <T$
\begin{equation*}
    \begin{split}
        N&\left[\sumrs[N](r+s)c\rs(\tau) + \int_0^\tau \sumrs[N]\br\rps c\rps + \bs\rsp c\rsp \dd \sigma\right]\\
         &= N\left[\sumrs[N] (r+s)c^0\rs + \int_0^\tau \sumrs[N] (\ar\rs c_{1,0}+\as\rs c_{0,1})c\rs \dd \sigma + \int_0^\tau N \sumrs[N-1][]\Jr\rs + \Js\rs\dd\sigma\right]\\
        &\le K\left( \sumrs[N](r+s)^2 c^0\rs + \int_0^\tau N\sumrs[N] (r+s)c\rs \dd \sigma + N^2 \sumrs[N] c\rs(\tau)-c\rs^0\right).
    \end{split}
\end{equation*}
But now we notice that $N^2 \sumrs[N]c\rs(\tau)\le \frac{N}{\vphi(N)}\sumrs[N]\vphi(r+s)(r+s)c\rs \overset{\text{Lemma \ref{mcBDApriori}}}\le K \frac{N}{\vphi(N)} \to 0$. So by Grönwall, we obtain the desired convergence to zero.

\end{proof}
\begin{remark}
  There is a very subtle difference to the one-component case. It is the boundary term $C_N$, where we put the absolute values inside the sum $N\sumrs[N][]|\Jr\rs| + |\Js\rs|.$ Proving that $N\big|\sumrs[N][] \Jr\rs + \Js\rs\big|$ goes to zero is much simpler, but since we do not have any control over the $\sgn\rps$ terms, this is not enough. Hence, we only find uniqueness for bounded second moments, whereas in the one-component case, if $a_r \in \bigO(r^\alpha)$, it suffices to have $\sum\limits_{n=1}^{\infty} n^{1+\alpha} c^0_n<\infty.$ This is---in spirit---the same problem we encounter, when trying to prove the weak formulation for general two-dimensional test sequences $g\rs$.
\end{remark}
\subsection{Entropy Equation}\label{sectionEntropyEquation}
Before we come to the entropy equation, one has to know that solutions are positive for all positive times.
    \begin{proposition}\textit{(Positivity)}\label{solutionPositive}\\
        Assume that $\ar\rs,\as\rs,\br\rps,\bs\rsp >0$ for all $r+s\ge 1$. Then, any solution $c$ to \eqref{mcBD} that satisfies 
\[\sumrs r c\rs(0) \neq 0 \neq \sumrs s c\rs(0),\]
already satisfies $c\rs(t) > 0$ for all $t>0$ and $r+s\ge 1.$

\end{proposition}
\begin{proof}
  This was also shown in \cite{dunwell:phd} following the arguments of \cite{penrose:foundations}.
\end{proof}
From here on out, we always assume detailed balance (Definition \ref{defiDetailedBalance}), i.e. $0<\ar\rs,\as\rs,Q\rs$ and $\br\rps = \ar\rs \frac{Q\rs}{Q\rps}$ and $\bs\rsp = \as\rs \frac{Q\rs}{Q\rsp}.$ The following Proposition tells us, that we can view these $Q\rs$ as given coefficients (replacing the $\br\rps,\bs\rsp$).

    \begin{proposition}\textit{(Detailed balance)}\label{detailedBalance}\\
        Let $\ar\rs,\as\rs,\br\rps,\bs\rsp >0$. Then, the following are equivalent
\begin{enumerate}
  \item $\ar\rs,\as\rs,\br\rps,\bs\rsp$ satisfy detailed balance.
  \item $\frac{\as_{1,0}}{\bs_{1,1}}=\frac{\ar_{0,1}}{\br_{1,1}}$ and $\frac{\ar\rs \as\rps}{\br\rps \bs\rpsp} = \frac{\as\rs \ar\rsp}{\bs\rsp \br\rpsp}$ for all $r+s\ge1$.
  \item There exist $Q\rs>0$ with $Q_{1,0} = 1 = Q_{0,1}$, such that 
    \[ \br\rps = \ar\rs \frac{Q\rs}{Q\rps} \quad\text{ and }\quad \bs\rsp = \as\rs \frac{Q\rs}{Q\rsp}.\]
\end{enumerate}

\end{proposition}
\begin{proof}
    i) $\implies$ ii): Let $\cbar\rs$ be an equilibrium, then we have
\begin{equation}
  \begin{split}
    \as_{1,0} \cbar_{1,0} \cbar_{0,1} - \bs_{1,1}\cbar_{1,1}
    = \Js_{1,0} = 0 = \Jr_{0,1}
    = \ar_{0,1} \cbar_{1,0}\cbar_{0,1} - \br_{1,1}\cbar_{1,1}
    \implies \frac{\as_{1,0}}{\bs_{1,1}} = \frac{\ar_{0,1}}{\br_{1,1}}.
  \end{split}
\end{equation}
Furthermore, for any $r+s\ge 1$, we obtain
\begin{equation}
    \cbar\rpsp \overset{\Js\rps = 0} = \cbar\rps \cbar_{0,1} \frac{\as\rps}{\bs\rpsp} 
    \overset{\Jr\rs = 0} = \cbar\rs \cbar_{1,0}\cbar_{0,1} \frac{\as\rps \ar\rs}{\bs\rpsp \br\rps}
\end{equation}
and 
\begin{equation}
    \cbar\rpsp \overset{\Jr\rsp = 0} = \cbar\rsp \cbar_{1,0} \frac{\ar\rsp}{\br\rpsp} 
    \overset{\Js\rs = 0} = \cbar\rs \cbar_{1,0}\cbar_{0,1} \frac{\ar\rsp \as\rs}{\br\rpsp \bs\rsp}.
\end{equation}
ii) $\implies$ iii): We can set $Q_{1,0} = 1 =Q_{0,1}$ and then iteratively define $Q\rps = Q\rs \frac{\ar\rs}{\br\rps}$ and $Q\rsp = Q\rs \frac{\as\rs}{\bs\rsp}$, which is exactly well defined by ii).\\
iii) $\implies$ i): Fix $\cbar_{1,0},\cbar_{0,1} >0$ and let $\cbar\rs = \cbar_{1,0}^r \cbar_{0,1}^s Q\rs$, which satisfies $\Jr\rs(\cbar) \equiv 0 \equiv \Js\rs(\cbar).$

\end{proof}
Next, we can introduce the entropy.
    \begin{notation}\textit{}\label{entropies}\\
        We will denote
\begin{equation*}
    \begin{split}
        V(c) &\coloneqq \sumrs c\rs \left[\ln\left(\frac{c\rs}{Q\rs}\right)-1\right],\\
        \Vn(c) &\coloneqq \sumrs[1][N] c\rs \left[\ln\left(\frac{c\rs}{Q\rs}\right)-1\right],\\
        \D(c) &\coloneqq - \sumrs \Jr\rs \left(\delr \ln\frac{c\rps}{Q\rps} -\ln{\monr}\right) + \Js\rs \left(\dels \ln\frac{c\rsp}{Q\rsp} -\ln{\mons}\right)\text{ and}\\
        \Dn(c) &\coloneqq - \sumrs[1][N-1] \Jr\rs \left(\delr \ln\frac{c\rps}{Q\rps} -\ln{\monr}\right) + \Js\rs \left(\dels \ln\frac{c\rsp}{Q\rsp} -\ln{\mons}\right).
    \end{split}
\end{equation*}

    \end{notation}
The continuity properties of $V$ heavily depend on $Q\rs$. As mentioned in the introduction, we will assume \eqref{assCoeff}, which we restate here for the readers convenience
\begin{equation}\tag*{\eqref{assCoeff}}
   \begin{split}
       \lim\limits_{N\to\infty}\inf\limits_{r+s\ge N} Q\rs^\frac 1 {r+s} >0 \quad\text{ and }\quad
       \lim\limits_{N\to\infty}\sup\limits_{r+s\ge N} Q\rs^\frac 1 {r+s} <\infty.
   \end{split}
\end{equation}
This yields the following continuity properties.
    \begin{lemma}\textit{(Strong continuity of $\Vzw$)}\label{entropyStrongCont}\\
        Let $\rho,\sigma>0$, then we have the following continuity properties.
\begin{itemize}
    \item The function $U\colon \Brohsigp \to \R$, $c\mapsto \sumrs c\rs\ln c\rs$ is continuous w.r.t. the weak* topology.
    \item Assume \eqref{assCoeff}, then the function $\Vzw \colon X^+ \to \R$ is continuous w.r.t. the strong topology.
\end{itemize}

    \end{lemma}
\begin{proof}
    For part one, notice that for any $0<\eps<1$, there is an $C_\eps$, s.t. $x|\ln x|\le C_\eps (x^{1+\eps} + x^{1-\eps})$ for all $x\ge 0$. Then, we calculate for any $N\in \N$, via $c\rs \le \frac{\rho+\sigma}{r+s}$
\begin{equation*}
    \begin{split}
        \sumrs[N]c\rs |\ln c\rs| &\le C_\eps \sumrs[N] c\rs^{1+\eps} + c\rs^{1-\eps} \\
                                 &\le C_\eps \sumrs[N] \frac{(\rho+\sigma)^{1+\eps}}{(r+s)^{1+\eps}} + C_\eps \sumrs[N] ((r+s)c\rs)^{1-\eps}(r+s)^{\eps-1}\\
        \overset{\text{H\"older}}&{\le}C_\eps \sumrs[N] \frac{(\rho+\sigma)^{1+\eps}}{(r+s)^{1+\eps}} + C_\eps \Bigg(\underbrace{\sumrs[N] (r+s)c\rs}_{\le \rho+\sigma}\Bigg)^{1-\eps}\left(\sumrs[N](r+s)^\frac{\eps-1}\eps\right)^\eps,
    \end{split}
\end{equation*}
which goes to zero independent of $c$, once we fix an $\eps$ satisfying $\frac{\eps-1}\eps < -1$. For part two, we already have that $\sumrs c\rs \ln c\rs -c\rs$ is even weak* continuous and due to \eqref{assCoeff} we find, that $|\ln Q\rs| \in \mathcal{O}(r+s)$, so 
\[\sumrs c\rs \ln Q\rs\]
is continuous with respect to the strong topology.

\end{proof}
When proving the entropy equation, we will again run into the trouble of treating the boundary term for $N$ in \eqref{weakmcBDfinite}.
To resolve the issue, we have to assume a higher moment bound on the initial conditions.
    \begin{theorem}\textit{(Entropy equation)}\label{entropyEquation}\\
        Assume that $\ar\rs,\as\rs,\br\rps,\bs\rsp \in \mathcal{O}((r+s)^\alpha)$, for some $0\le \alpha <1$ and assumption \eqref{assCoeff}. Furthermore, assume $\rho(c^0) \neq 0 \neq \sigma(c^0)$ and $\sumrs (r+s)^{1+\alpha} c^0\rs < \infty.$ Then, any solution $c \colon [0,T)\to X^+$ to \eqref{mcBD} will satisfy 
\begin{equation*}
    \Vzw(c(t)) = \Vzw(c(0)) - \int_0^t D(c(\tau)) \dd \tau \quad\text{ for all }0\le t < T.
\end{equation*}

    \end{theorem}
\begin{proof}
    We start by noticing, that due to our bounds on the coefficients we find
\[\sumrs |\Jr\rs| + |\Js\rs| \le C(\norm{c(0)}) \sumrs (r+s)^\alpha c\rs, \]
which converges uniformly, so we can differentiate $\monr$ and $\mons$ almost everywhere. Furthermore, we can differentiate $c\rs (\ln\frac{c\rs}{Q\rs}-1)$ for positive times due to Proposition \ref{solutionPositive} and obtain for almost every $t\in(0,T)$,
\begin{equation*}
    \begin{split}
      \odv{}{t}\Vzwn(c(t)) \overset{\eqref{weakmcBDfinite}}&{=} -\Dnp(c(t)) - \sumrs[N][] \Jr\rs \ln\frac{c\rps}{Q\rps} + \Js\rs \ln\frac{c\rsp}{Q\rsp} - \sumrs[N+1] \Jr\rs \ln \monr + \Js\rs \ln\mons\\
       &= -\Dn(c(t)) - \sumrs[N-1][] \Jr\rs \ln\frac{c\rs}{Q\rs} + \Js\rs \ln\frac{c\rs}{Q\rs} - \sumrs[N] \Jr\rs \ln \monr + \Js\rs \ln\mons.
    \end{split}
\end{equation*}
In order to take $N\to\infty$, we need some bounds on the remainder terms. First we notice, that due to $0<c\rs<\frac{\norm{c(0)}}{r+s}<1$ for large $r+s$, we have $\ln c\rs <0$ for large $r+s$. Consequently, we obtain for $N$ large enough
\begin{equation*}
    \begin{split}
        -\sumrs[N][]\Jr\rs \ln c\rps + \Js\rs \ln c\rsp
        \ge \sumrs[N][] \br\rps c\rps \ln c\rps + \bs\rsp c\rsp \ln c\rsp \eqqcolon (*).
    \end{split}
\end{equation*}
Now we use that for $0\le x \le 1$ we find for every $0<\eps<1$ a constant such that $x|\ln x| \le C_\eps x^{1-\eps}$, to find
\begin{equation*}
    \begin{split}
        |(*)| \overset{\text{ass $\br,\bs$}}&{\le} C(N+1)^\alpha \sumrs[N+1][] c\rs|\ln c\rs|
        \le C_\eps (N+1)^\alpha \sumrs[N+1][] c\rs^{1-\eps}\\
                                            &= C_\eps (N+1)^\alpha \sumrs[N+1][] (N+1)^{\eps-1}((N+1)c\rs)^{1-\eps}\\
        \overset{\text{H\"older}}&{\le} C_\eps (N+1)^{\alpha+2\eps-1} \left(\sumrs[N+1][] (N+1)c\rs\right)^{1-\eps}.
    \end{split}
\end{equation*}
Taking $\eps$ so small, that $\alpha+2\eps < 1$, we obtain $|(*)| \to 0$ as $N\to\infty$ uniformly in $t$. Similarly, we find
\begin{equation*}
    -\sumrs[N-1][] \Jr\rs \ln c\rs + \Js\rs \ln c\rs \le - \sumrs[N-1][] \ar\rs \monr c\rs \ln c\rs + \as\rs \mons c\rs \ln c\rs \to 0 \text{ uniformly in }t,
\end{equation*}
with the same argument as for $(*)$.
Next, we use $|\ln Q\rs| \in \bigO(r+s)$, to find for some constant $K<\infty$
\begin{equation*}
    \begin{split}
        \int_{t_1}^{t_2} \sumrs[N][] |\Jr\rs| |\ln  Q\rps| + |\Js\rs| |\ln Q\rsp |\dd t
        \le K \int_{t_1}^{t_2} \sumrs[N] (r+s)^{1+\alpha} c\rs \dd \tau \overset{\text{Lemma \ref{mcBDApriori}}} \to 0 \text{ as } N\to \infty.
    \end{split}
\end{equation*}
Again the same argument allows to conclude
\[\left|\int_{t_1}^{t_2} \sumrs[N-1][] \Jr\rs \ln  Q\rs + \Js\rs \ln Q\rs \dd t \right| \to 0 \text{ as }N\to \infty.\]
Finally, we find 
\begin{equation*}
    \begin{split}
        &\left|\int_{t_1}^{t_2}\sumrs[N+1]\Jr\rs \ln \monr + \Js\rs\ln\mons \dd t\right| \\
        &\quad\le \underbrace{\sup\limits_{t_1\le t\le t_2}|\ln\monr |}_{\le C} \int_{t_1}^{t_2} \sumrs[N+1]|\Jr\rs|\dd t + \underbrace{\sup\limits_{t_1\le t\le t_2}|\ln\mons |}_{\le C} \int_{t_1}^{t_2} \sumrs[N+1] |\Js\rs| \dd t \to 0 \text{ as }N\to\infty.
    \end{split}
\end{equation*}
Putting everything together, we have shown
\begin{equation*}
    \begin{split}
        -\int_{t_1}^{t_2} \Dn(c(t)) \dd t + o(1) \le \Vzwn(c(t_2)) - \Vzwn(c(t_1)) \le -\int_{t_1}^{t_2} \Dnp(c(t)) \dd t + o(1).
    \end{split}
\end{equation*}
Here the left and right hand side converge, due to monotone convergence, as $N\to\infty$ whereas $\Vzwn$ converges due to Lemma \ref{entropyStrongCont}, once we noticed that $\Vzwn(c) = \Vzw(c\mathds{1}_{\{r+s\le N\}})$. Lastly, as we take $t_1\to 0$, we know by Proposition \ref{solutionContinuity} that $c(t_1)\to c(0)$ in $X^+$. And by Lemma \ref{entropyStrongCont} we find $\Vzw(c(t_1))\to \Vzw(c(0))$. Exploiting once more monotone convergence for the integral we arrive at
\begin{equation*}
    \Vzw(c(t_2)) = \Vzw(c(0)) - \int_0^{t_2} \D(c(t))\dd t. \qedhere
\end{equation*}

\end{proof}
\begin{remark}
  One can quickly check that under the assumption $\br\rps,\bs\rsp \in \bigO((r+s)^\alpha)$ it suffices to ask for $\sumrs (r+s)^{1+\alpha} c^0\rs < \infty$ to carry out the uniqueness proof. The boundary term, that was denoted $C_N$, can then be estimated as in the proof of Theorem \ref{entropyEquation}. Hence, we may replace the ``any solution'' by ``the solution'' in Theorem \ref{entropyEquation}.
\end{remark}
Lastly, we want to show that we can always find a solution such that $\odv{}{t} V(c(t)) \le -D.$
    \begin{theorem}\textit{}\label{entropyInequality}\\
        Let $c^0\in X^+$ with $\rho(c^0) \neq 0 \neq \sigma(c^0)$ and assume that $\ar\rs + \as\rs \in \mathcal{O}(r+s),$ as well as \eqref{assCoeff}.
Then, there exists a solution $c$ to \eqref{mcBD} with initial data $c^0$ on $[0,\infty)$ satisfying 
\begin{equation*}
  \Vzw(c(t_2)) \le \Vzw(c(t_1)) - \int_{t_1}^{t_2} D(c(\tau)) \dd \tau \quad\text{ for all }0\le t_1 < t_2 < \infty.
\end{equation*}

    \end{theorem}
\begin{proof}
    As seen in the proof of Theorem \ref{entropyEquation}  we can assume without loss of generality $c^0\rs >0$ for all $r,s$. Let $c^n$ be the solution to the cut-off system as in the proof of Theorem \ref{mcBDExistenceDela}. Then, for any $0\le t_1<t_2<\infty$
\begin{equation}
  V^n(c^n(t_2)) = V^n(c^n(t_1)) - \int_{t_1}^{t_2} D^n(c^n(\tau))\dd\tau
\end{equation}
holds true. From the proof of Theorem \ref{mcBDExistenceDela}, we have up to some subsequence, $c^n(t) \weakstar c(t),$ where $c(t)$ is a solution to \eqref{mcBD}. But due to Lemma \ref{mcBDApriori} we actually have $c^n(t) \to c(t)$ in $X$. Therefore, Lemma \ref{entropyStrongCont} implies $V^n(c^n(t_1)) \to V(c(t_1))$ and $V^n(c^n(t_2))\to V(c^n(t_2))$. Because $c^n$ and $c$ are bounded from above and below, we finally have for any $m\in\N$
\begin{equation}
  \liminf_{n\to\infty} \int_{t_1}^{t_2} D^n(c^n) \dd\tau
  \ge \liminf_{n\to\infty} \int_{t_1}^{t_2} D^m(c^n) \dd\tau
  = \int_{t_1}^{t_2} D^m(c) \dd\tau.
\end{equation}
Since this bound holds for any $m$ we find $\displaystyle \liminf_{n\to\infty} \int_{t_1}^{t_2} D^n(c^n) \dd\tau \ge \int_{t_1}^{t_2} D(c) \dd\tau$.

\end{proof}

\section{Steady States}\label{sectionSteadyStates}
Due to Theorem \ref{massConservation} we know  $c(t) \subset B^+_{\rho,\sigma}$, which is compact in the weak* topology according to Proposition \ref{weakMetric}. Hence, we know that along some subsequence $c(t_j)\weakstar d$ for some $d\in B^+_{\rho,\sigma}$. It is possible to deduce from Theorem \ref{entropyEquation}, that all possible limits must satisfy $D(d) = 0$. Now, for a positive sequence $c\rs$, let us denote $\cbar\rs \coloneqq \monr^r \mons^s Q\rs$ and $u\rs \coloneqq \frac{c\rs}{\cbar\rs},$ then we find 
\begin{equation}
  \begin{split}
    \Jr\rs(c) \overset{\text{\eqref{fluxes} and \eqref{detailedBalance}}}&= \ar\rs\monr c\rs - \ar\rs \frac{Q\rs}{Q\rps} c\rps = - \ar\rs \monr \cbar\rs \delr u\rps \text{ and }\\
    \Js\rs(c) \overset{\text{\eqref{fluxes} and \eqref{detailedBalance}}}&= \as\rs\mons c\rs - \as\rs \frac{Q\rs}{Q\rsp} c\rsp = - \as\rs \mons \cbar\rs \dels u\rsp.
  \end{split}
\end{equation}
Putting this into the definition of $\D$ (Notation \ref{entropies}) yields 
\begin{equation}\label{Dviau}
  \D(c) =\sumrs \ar\rs \monr \cbar\rs \Big(\delr u\rps \delr \ln(u\rps)\Big)
         + \as\rs \mons \cbar\rs \Big(\dels u\rsp \dels \ln(u\rsp)\Big).
\end{equation}
By the monotonicity of the logarithm, $D(c) = 0$ implies $\delr u\rps = 0 =\dels u\rsp$. Since $Q_{1,0} = Q_{0,1} =1$, we find $u\rs \equiv 1.$ Hence, all limit points must be of the form $d\rs = z^r w^s Q\rs \eqqcolon \cinfzw\rs$. Given the mass constraint $\cinfzw \in B^+_{\rho,\sigma},$ it is sensible to study these steady states. We will first do so on finite sub-regions of $\Omega$ (say $\Gamma$), where for all $\rho,\sigma>0$, we can find a unique steady state, that attains these masses. Then, we consider the limit $\Gamma \to \Omega$, which yields a nice description of all steady states with finite mass (Theorem \ref{steadyStateLimit}). When establishing the limit, the main tool will be the relative entropy between steady states.
To start, let us fix the following Notation.
    \begin{notation}\textit{}\label{Gamma}\\
        We will denote with $\Gamma$ a subset of $\Omega$ satisfying 
\[ \Gamma \subset \Omega \text{ such that }(1,0),(0,1) \in \Gamma \text{ and }|\Gamma|<\infty.\]
Furthermore, we will denote for any $z,w\ge 0$ 
\[\cbargamma \coloneqq \left(z^r w^s Q\rs\right)_{r,s \in \Gamma}.\]
Finally, for a sequence of regions $\Gamma_n,$ we say $\Gamma_n \to \Omega$, if for any $(r,s)\in \Omega,$ there is an $M\in\N$, such that for all $n\ge M$, we have $(r,s)\in \Gamma_n.$

    \end{notation}
    \begin{lemma}\textit{(Unique finite steady state with given mass)}\label{finiteSteadyStateDB}\\
        Fix $\rho,\sigma \ge 0$ and $\Gamma$. Then, there exists unique $z^\Gamma,w^\Gamma \ge 0,$ such that 
\begin{equation*}
  \rho^\Gamma(\cbargamma) \coloneqq \sumgamma r \cbargamma\rs = \rho 
  \qquad\text{ and }\qquad
  \sigma^\Gamma(\cbargamma) \coloneqq \sumgamma s \cbargamma\rs = \sigma 
\end{equation*}

    \end{lemma}
\begin{proof}
  For fixed $z$ the map
\begin{equation*}
  w \mapsto \sumgamma s z^r w^s Q\rs
\end{equation*}
is continuous and strictly monotone with $\sigma^\Gamma(z^r 0^s Q\rs) = 0$ and $\sigma^\Gamma(z^r w^s Q\rs) \to \infty $ for $w \to \infty$. So we can find a unique $w_\sigma(z)$, such that  $\sumgamma s z^r w_\sigma(z)^s Q\rs = \sigma.$ The same way, we can define up to some maximal value of $z$, say $\bar z$, the function $w_\rho(z)$, that satisfies $\sumgamma r z^r w_\rho(z)^s Q\rs = \rho$.
In particular, we have 
\begin{equation*}
  w_\rho(0) = \infty,\quad w_\rho(\bar z) = 0,\quad w_\sigma(0)<\infty \quad \text{ and }\quad w_\sigma(\bar z)>0.
\end{equation*}
Since, we are looking for an intersection of $w_\rho$ and $w_\sigma$, we will finish the proof, by showing that they are continuous and the difference is monotonically decreasing. By the implicit function theorem we have 
\begin{equation*}
  \odv{w_\rho}{z} = - \frac 1 {\odv{\rho^\Gamma} w } \odv {\rho^\Gamma} z = -\frac w z \frac {\sumgamma r^2 z^r w^s Q\rs}{\sumgamma rs z^r w^s Q\rs} 
  \quad\text{ and }\quad
  \odv{w_\sigma}{z} = - \frac 1 {\odv{\sigma^\Gamma} w } \odv {\sigma^\Gamma} z = -\frac w z \frac {\sumgamma rs z^r w^s Q\rs}{\sumgamma s^2 z^r w^s Q\rs}.
\end{equation*}
In particular these functions are continuous, and we have 
\begin{equation*}
   \odv{w_\rho}{z} -  \odv{w_\sigma}{z}
   = \frac w z \frac {\left(\sumgamma r s z^r w^s Q\rs\right)^2 - \left(\sumgamma r^2 z^r w^s Q\rs\right)\left(\sumgamma s^2 z^r w^s Q\rs\right)}{\left(\sumgamma s^2 z^r w^s Q\rs\right)\left(\sumgamma rs z^r w^s Q\rs\right)} \overset{\text{C.S.}}< 0.
\end{equation*}
 
\end{proof}
The first step in taking $\Gamma \to \Omega$ is the following definition.
    \begin{definition}\textit{(Region of existence)}\label{regionOfExistence}\\
        Let us denote $\cinfzw\rs \coloneqq z^r w^s Q\rs$. Then, we define the region of existence as 
\[\Ex \coloneqq\left\{(z,w) \mid z\ge 0, w\ge 0 \text{ and }\rho(\cinfzw) + \sigma(\cinfzw)<\infty\right\} .\]

\end{definition}
Note, that for all $(z,w)\in\Ex$ the sequence $\cinfzw\rs$ is a steady state to \eqref{mcBD}, because they satisfy $\Jr\rs \equiv 0 \equiv \Js\rs$, as seen in the introduction of this section. The form of $\Ex$ depends on $Q\rs$ and can have very different shapes, as we can see in the following examples.
    \begin{example}\textit{}\label{steadyStateDBExample}\\
        In the following examples, we always consider $r+s$ large. In particular, we do not need to worry about $Q_{1,0} = Q_{0,1}$.
\begin{itemize}
  \item \underline{Idealised Kelvin model:} $\displaystyle Q\rs = \lambda^r \mu^s \binom{r+s}{r}e^{-(r+s)^{\frac 2 3}}$\\ 
    This is a simple model describing droplet formation \cite{wilemski:Coefficients}. Here  $0<\lambda,\mu<\infty$ are given physical constants. For $0\le z,w < \infty$ and a large $M$ we find
    \[ \sumrs[M] (r+s) \cinfzw\rs = 
    \sum\limits_{n=M}^{\infty} n e^{-n^{\frac 2 3}} \sumrs[n][] (\lambda z)^r (\mu w)^s \binom{r+s}{r} = \sum\limits_{n=M}^{\infty} n (\lambda z + \mu w)^n e^{-n^{\frac 2 3}}.\]
    So we can determine $\Ex = \{\lambda z + \mu w \le 1\}$.

  \item \underline{$\partial \Ex \not \in C^1$:} $\displaystyle Q\rs = \lambda^r \mu^s e^{-(r+s)^{\frac 2 3}}$\\
    The part of the boundary, that is not coming from $z\ge0$ and $w\ge 0$ does not need to be smooth. With fixed $0<\lambda,\mu<\infty$ we see for $0<z,w<\infty$
    \[ \sumrs[M] (r+s) \cinfzw\rs Q\rs = \sumrs[M] (r+s) (\lambda z)^r (\mu w)^s e^{-(r+s)^{\frac 2 3}} < \infty \iff \max(\lambda z ,\mu w)\le 1.\]
  \item \underline{$\Ex$ not compact:} $\displaystyle Q\rs = \frac {e^{-r^{\frac 2 3}}} {s!}  $\\
    Without the assumption \eqref{assCoeff}, $\Ex$ does not need to be bounded. If we denote $f(z) \coloneqq \sum\limits_{r=0}^{\infty} z^r e^{-r^\frac 2 3},$ we obtain for $0<z,w<\infty$
    \[\sumrs s z^r w^s Q\rs = f(z) w e^w \quad\text{ and }\quad \sumrs r z^r w^s Q\rs = z f^\prime(z) e^w\]
    and hence $\Ex = \{z\le 1\}.$
\end{itemize}
    We have displayed $\Ex$ as well as $(\rho(\cinfzw),\sigma(\cinfzw))$ for all $(z,w)\in \Ex$ in Figure \ref{figure:examplesDBE}.

    \end{example}
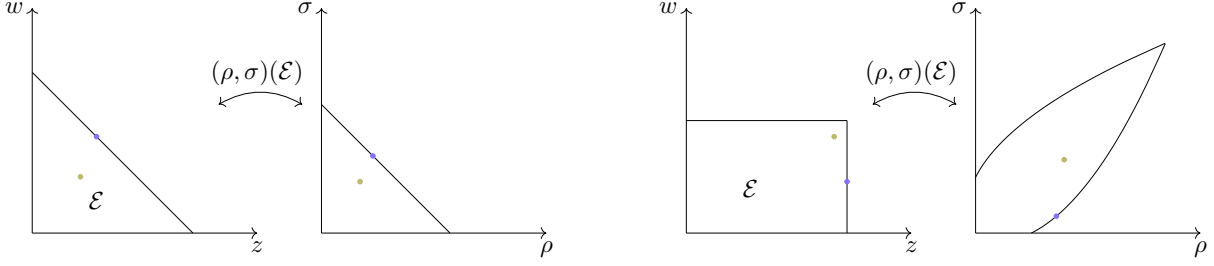
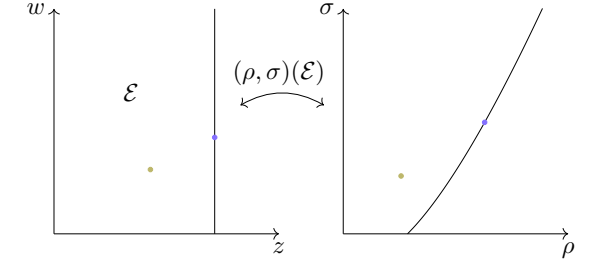
\begin{figure}[!htb]
  \centering
  \begin{subfigure}[t]{0.45\textwidth}
    \scalebox{0.85}{
      \begin{tikzpicture}
        \begin{scope}[shift={(0,0)}]
          \node at (1,0.5) (a) {$\Ex$}; 
          \draw[->] (0, 0) -- (3.5, 0) node[below] {$z$};
          \draw[->] (0, 0) -- (0, 3.5) node[left] {$w$};
          \draw[scale=1, domain=0:2.5, smooth, variable=\x, black] plot ({\x}, {2.5-\x});
          \filldraw[LightSlateBlue] (1,1.5) circle (1pt);
          \filldraw[DarkKhaki] (0.75,0.875) circle (1pt);
          \draw[<->] (2.9,2) to[bend left=30] (4.2,2);
          \node at (3.5,2.5) (b) {$(\rho,\sigma)(\Ex)$};
        \end{scope}
        \begin{scope}[shift={(4.5,0)}]
            \draw[->] (0, 0) -- (3.5, 0) node[below] {$\rho$};
            \draw[->] (0, 0) -- (0, 3.5) node[left] {$\sigma$};
            \draw[scale=1, domain=0:2, smooth, variable=\x, black] plot ({\x}, {2-\x});
            \filldraw[LightSlateBlue] (0.8,1.2) circle (1pt);
            \filldraw[DarkKhaki] (0.6,0.8) circle (1pt);
          \end{scope}
      \end{tikzpicture}}
    \caption{The idealised Kelvin model, a physically relevant example: $Q\rs = \lambda^r \mu^s \binom{r+s}{r}e^{-(r+s)^{\frac 2 3}}$.}
  \end{subfigure}
  \hfill
    \begin{subfigure}[t]{0.45\textwidth}
      \scalebox{0.85}{
      \begin{tikzpicture}[
              scale=1,
              declare function={
                f(\x,\k) = 
                  \x^1*1*exp(-(1+\k)^(2/3))+
                  \x^2*2*exp(-(2+\k)^(2/3))+
                  \x^3*3*exp(-(3+\k)^(2/3))+
                  \x^4*4*exp(-(4+\k)^(2/3))+
                  \x^5*5*exp(-(5+\k)^(2/3))+
                  \x^6*6*exp(-(6+\k)^(2/3))+
                  \x^7*7*exp(-(7+\k)^(2/3))+
                  \x^8*8*exp(-(8+\k)^(2/3))+
                  \x^9*9*exp(-(9+\k)^(2/3))+
                  \x^10*10*exp(-(10+\k)^(2/3))+
                  \x^11*11*exp(-(11+\k)^(2/3))+
                  \x^12*12*exp(-(12+\k)^(2/3))+
                  \x^13*13*exp(-(13+\k)^(2/3))+
                  \x^14*14*exp(-(14+\k)^(2/3))+
                  \x^15*15*exp(-(15+\k)^(2/3))+
                  \x^16*16*exp(-(16+\k)^(2/3))+
                  \x^17*17*exp(-(17+\k)^(2/3))+
                  \x^18*18*exp(-(18+\k)^(2/3))+
                  \x^19*19*exp(-(19+\k)^(2/3))+
                  \x^20*20*exp(-(20+\k)^(2/3))+
                  \x^21*22*exp(-(21+\k)^(2/3))+
                  \x^22*22*exp(-(22+\k)^(2/3));
                rho(\z,\w) = 
                  \w^0 * f(\z,0)+
                  \w^1 * f(\z,1)+
                  \w^2 * f(\z,2)+
                  \w^3 * f(\z,3)+
                  \w^4 * f(\z,4)+
                  \w^5 * f(\z,5)+
                  \w^6 * f(\z,6)+
                  \w^7 * f(\z,7)+
                  \w^8 * f(\z,8)+
                  \w^9 * f(\z,9)+
                  \w^10 * f(\z,10)+
                  \w^11 * f(\z,11)+
                  \w^12 * f(\z,12)+
                  \w^13 * f(\z,13)+
                  \w^14 * f(\z,14)+
                  \w^15 * f(\z,15)+
                  \w^16 * f(\z,16)+
                  \w^17 * f(\z,17)+
                  \w^18 * f(\z,18)+
                  \w^19 * f(\z,19)+
                  \w^20 * f(\z,20)+
                  \w^21 * f(\z,21)+
                  \w^22 * f(\z,22);
                sigma(\z,\w) = rho(\w,\z);
              }
          ]
          \begin{scope}[shift={(0,0)}]
            \node at (1,0.7) (a) {$\Ex$}; 
            \draw[->] (0, 0) -- (3.5, 0) node[below] {$z$};
            \draw[->] (0, 0) -- (0, 3.5) node[left] {$w$};
            \draw[scale=1, domain=0:2.5, smooth, variable=\x, black] plot ({\x}, {1.75});
            \draw[scale=1, domain=0:1.75, smooth, variable=\y, black] plot ({2.5}, {\y});
            \filldraw[LightSlateBlue] (2.5,0.8) circle (1pt);
            \filldraw[DarkKhaki] (2.3,1.5) circle (1pt);
            \draw[<->] (2.9,2) to[bend left=30] (4.2,2);
              \node at (3.5,2.5) (b) {$(\rho,\sigma)(\Ex)$};
          \end{scope}
          \begin{scope}[shift={(4.5,0)}]
              \draw[->] (0, 0) -- (3.5, 0) node[below] {$\rho$};
              \draw[->] (0, 0) -- (0, 3.5) node[left] {$\sigma$};
              \draw[scale=0.3, domain=0:1, smooth, variable=\x, black] plot ({rho(\x,1)}, {sigma(\x,1});
              \draw[scale=0.3, domain=0:1, smooth, variable=\y, black] plot ({rho(1,\y)}, {sigma(1,\y)});
              \filldraw[LightSlateBlue] ({0.3*rho(1,0.8/1.75)},{0.3*sigma(1,0.8/1.75)}) circle (1pt);
              \filldraw[DarkKhaki] ({0.3*rho(2.3/2.5,1.5/1.75)},{0.3*sigma(2.3/2.5,1.5/1.75)}) circle (1pt);
          \end{scope}
      \end{tikzpicture}}
    \caption{Dropping the binomial coefficient in a) leads to $Q\rs = \lambda^r \mu^s e^{-(r+s)^{\frac 2 3}}$.}
  \end{subfigure}
  \hfill
  \vspace{10pt}
  \begin{subfigure}{0.45\textwidth}
  \scalebox{0.85}{
  \begin{tikzpicture}[
                  declare function={
                      xln(\x) = ifthenelse(\x == 0, 0, \x*ln(\x));
                    }]
      \begin{scope}[shift={(0,0)}]
        \node at (1.2,2.2) (a) {$\Ex$}; 
        \draw[->] (0, 0) -- (3.5, 0) node[below] {$z$};
        \draw[->] (0, 0) -- (0, 3.5) node[left] {$w$};
        \draw[scale=1, domain=0:3.5, smooth, variable=\y, black] plot ({2.5}, {\y});
        \filldraw[LightSlateBlue] (2.5,1.5) circle (1pt);
        \filldraw[DarkKhaki] (1.5,1.0) circle (1pt);
        \draw[<->] (2.9,2) to[bend left=30] (4.2,2);
        \node at (3.5,2.5) (b) {$(\rho,\sigma)(\Ex)$};
      \end{scope}
      \begin{scope}[shift={(4.5,0)}]
        \draw[->] (0, 0) -- (3.5, 0) node[below] {$\rho$};
        \draw[->] (0, 0) -- (0, 3.5) node[left] {$\sigma$};
        \draw[scale=1, domain=1:3.1, smooth, variable=\x, black] plot ({\x}, {xln(\x)});
        \filldraw[LightSlateBlue] (2.2,{xln(2.2)}) circle (1pt);
        \filldraw[DarkKhaki] (0.9,0.9) circle (1pt);
      \end{scope}
  \end{tikzpicture}}
    \caption{An example violating assumption \eqref{assCoeff}, but still applicable to Section \ref{sectionSteadyStates}: $\displaystyle Q\rs = \frac {e^{-r^{\frac 2 3}}} {s!}$.}
  \end{subfigure}
  \caption{We have plotted $\Ex$ for the three examples given in \ref{steadyStateDBExample}. Furthermore, we have indicated for which $(\rho,\sigma)$, there is a corresponding equilibrium attaining these masses.}
  \label{figure:examplesDBE}
\end{figure}%

Coming back to the general study, if we fix $\rho,\sigma$ and take $\Gamma \to \Omega$, then Fatou tells us that the finite steady states $z^\Gamma, w^\Gamma$ will be close to $\Ex$. However, we can exploit the following convexity structure to analyse the limit precisely.
    \begin{definition}\textit{(Relative entropy between steady states)}\label{steadyStatesEntropy}\\
        For $z,w,u,v > 0$ and $\Gamma$ we define
\begin{equation*}
  H^\Gamma[u,v|z,w] \coloneqq \sumgamma \cinfzw\rs \Psi\left(\frac{\cinfuv\rs}{\cinfzw\rs}\right),\quad \text{ where } \Psi(x) = x \ln x + 1 -x.
\end{equation*}
Since $\Psi(x) \ge 0$, we get in particular
\begin{equation*}
  0 \le H^\Gamma[u,v|z,w] = \ln\left( \frac u z \right)\rho^\Gamma(\cinfuv) + \ln\left( \frac v w\right) \sigma^\Gamma(\cinfuv) + \sumgamma \cinfzw\rs -\cinfuv\rs.
\end{equation*}
Lastly, for $0< z,w,u,v$ with $(z,w),(u,v) \in \Ex$ we may also define 
\begin{equation*}
  H[u,v|z,w] \coloneqq\ln\left( \frac u z \right)\rho(\cinfuv) + \ln\left( \frac v w\right) \sigma(\cinfuv) + \sumrs \cinfzw\rs -\cinfuv\rs.
\end{equation*}

\end{definition}
This immediately yields, the following uniqueness in $\Ex$ and lets us determine the limit $\Gamma \to \Omega$.
    \begin{corollary}\textit{(Unique steady state in $\Ex$)}\label{steadyStateDB}\\
        For each $\rho,\sigma$ there is at most one $(z,w) \in \Ex$, with $\rho(\cinfzw) = \rho$ and $\sigma(\cinfzw) = \sigma$.

    \end{corollary}
\begin{proof}
  For each $\rho,\sigma$ there is at most one $(z,w) \in \Ex$, with $\rho(\cinfzw) = \rho$ and $\sigma(\cinfzw) = \sigma$.
 
\end{proof}
    \begin{theorem}\textit{}\label{steadyStateLimit}\\
        Assume $Q\rs$ is such that there exist $0< u,v$ with $(u,v) \in \Ex$. Then, for any fixed $\rho_0,\sigma_0 > 0$, we get $z^{\Gamma_n},w^{\Gamma_n} \to z,w>0 $ as $\Gamma_n \to \Omega$, where the pair $(z,w)\in \Ex$ is uniquely determined by 
\begin{equation}\label{steadyStateLimitEq}
  \ln\left(\frac z u\right)\left(\rho_0 - \rho(\cinfzw)\right) + \ln\left(\frac w v\right) \left(\sigma_0 - \sigma(\cinfzw)\right) \ge 0 \text{ for all } 0< u,v \text{ with }(u,v)\in \Ex.
\end{equation}

    \end{theorem}
\begin{proof}
  We start by bounding $z^\Gamma$ and $w^\Gamma$ independent of $\Gamma$ from below. Let $0<u,v$ with $(u,v) \in\Ex$. From $z^\Gamma \le \rho_0$ and $w^\Gamma \le \sigma_0$ we find
\begin{equation*}
  \begin{split}
    0 &\le H^\Gamma[z^\Gamma,w^\Gamma|u,v] = \rho_0 \ln \frac {z^\Gamma}{u} + \sigma_0 \ln \frac{w^\Gamma} w + \sumgamma \cinfuv\rs - \cbargamma\rs \\
    &\le 
    \begin{cases}
      \rho_0 \ln \frac{z^\Gamma}{u} + \sigma_0 \ln\frac{\sigma_0}{w} + \sumrs \cinfuv\rs \to -\infty \text{ as } z^{\Gamma} \to 0.\\
      \rho_0 \ln \frac{\rho_0}{u} + \sigma_0 \ln\frac{w^\Gamma}{w} + \sumrs \cinfuv\rs \to -\infty \text{ as } w^{\Gamma} \to 0.
    \end{cases}
  \end{split}
\end{equation*}

Because $z^{\Gamma_n},w^{\Gamma_n}$ are bounded from above and below we can extract a subsequence (not relabelled) that converges to some $0<z,w$ with $(z,w) \in\Ex$. Fix an $M\in\N$ and let $N\in \N$ be such that $\Gamma_M \subset \Gamma_n$ for all $n\ge N.$ We then find for any $0<u,v$ with $(u,v)\in\Ex$
\begin{equation*}
  \begin{split}
    H^{\Gamma_M}[z,w|u,v] \overset{n\to\infty}&\leftarrow H^{\Gamma_M}[z^{\Gamma_n},w^{\Gamma_n}|u,v] \overset{\Psi \ge 0} \le H^{\Gamma_n}[z^{\Gamma_n},w^{\Gamma_n}|u,v] \\
    &= \rho_0 \ln\left(\frac {z^{\Gamma_n}}{u}\right) + \sigma_0 \ln\left(\frac{w^{\Gamma_n}}{v}\right) + \sumgamma[\Gamma^n] \cinfuv\rs - \cbar^{\Gamma_n}\rs\\
    &= H^{\Gamma_n}[z,w|u,v]
    + \rho_0 \ln\left(\frac{z^{\Gamma_n}}z\right) + \sigma_0\ln\left(\frac{w^{\Gamma_n}}w\right) + \sumgamma[\Gamma^n] \cinfzw\rs - \cbar^{\Gamma_n}\rs\\
    &\qquad\qquad+ \ln\left(\frac z u\right)\left(\rho_0 - \rho^{\Gamma_n}(\cinfzw)\right)
    + \ln\left(\frac w v\right)\left(\sigma - \sigma^{\Gamma_n}(\cinfzw)\right)\\
    \overset{n\to\infty}&\to H[z,w|u,v] + \ln\left(\frac z u\right)\left(\rho_0 - \rho(\cinfzw)\right)
    + \ln\left(\frac w v\right)\left(\sigma - \sigma(\cinfzw)\right).
  \end{split}
\end{equation*}
Since this bound holds for any $M$, we find 
\begin{equation*}
  \ln\left(\frac z u\right)\left(\rho_0 - \rho(\cinfzw)\right) + \ln\left(\frac w v\right) \left(\sigma_0 - \sigma(\cinfzw)\right) \ge 0 \text{ for all } 0< u,v \text{ with }(u,v) \in\Ex.
\end{equation*}
It remains to show, that this determines a unique steady state. For that purpose, we first show $H^{\Gamma_n}[z^{\Gamma_n},w^{\Gamma_n}|z,w] \to 0$ as $n\to \infty.$ Note that $c^{\Gamma_n}\rs \indicator{(r,s) \in \Gamma_n} \weakstar \cinfzw$ (Proposition \ref{weakTopology}). And so, by Proposition \ref{weakMetric}, we obtain 
\[\sum\limits_{r,s\in \Gamma^n} c^{\Gamma_n} \to \sumrs \cinfzw.\]
From this $H^{\Gamma_n}[z^{\Gamma_n},w^{\Gamma_n}|z,w] \to 0$ follows, because $\sum\limits_{r,s\in \Gamma^n} \cinfzw\rs \to \sumrs \cinfzw.$
Now, assume $0<z_2,w_2$ with $(z_2,w_2)\in\Ex$ also has property \eqref{steadyStateLimitEq}. Then we find 
\begin{equation*}
  \begin{split}
    0 \overset{n\to\infty}&\leftarrow H^{\Gamma_n}[z^{\Gamma_n},w^{\Gamma_n}|z,w] = H^{\Gamma_n}[z^{\Gamma_n},w^{\Gamma_n}|z_2,w_2] + H^{\Gamma_n}[z_2,w_2|z,w]\\
    &\qquad\qquad\qquad\qquad\qquad+ \ln\left(\frac {z_2}z\right) (\rho_0 - \rho^{\Gamma_n}(\cbar^{z_2,w_2})) +  \ln\left(\frac {w_2}w\right) (\rho_0 - \sigma^{\Gamma_n}(\cbar^{z_2,w_2}))\\
    &\ge H^{\Gamma_n}[z_2,w_2|z,w]+ \ln\left(\frac {z_2}z\right) (\rho_0 - \rho^{\Gamma_n}(\cbar^{z_2,w_2})) +  \ln\left(\frac {w_2}w\right) (\rho_0 - \sigma^{\Gamma_n}(\cbar^{z_2,w_2}))\\
    \overset{n\to\infty}&\to H[z_2,w_2|z,w]+ \ln\left(\frac {z_2}z\right) (\rho_0 - \rho(\cbar^{z_2,w_2})) +  \ln\left(\frac {w_2}w\right) (\rho_0 - \sigma(\cbar^{z_2,w_2}))\\
    \overset{\text{Ass. on }z_2,w_2}&\ge H[z_2,w_2|z,w]\ge 0.
  \end{split}
\end{equation*}
So $H[z_2,w_2|z,w] = 0$ and therefore $z_2 = z$ and $w_2 = w.$
 
\end{proof}
Note, that assumption \eqref{assCoeff} implies the existence of $0<u,v$ with $(z,w) \in\Ex$, so Theorem \ref{steadyStateLimit} captures in particular the situation of Theorem \ref{entropyEquation}. Even though it is not applicable under assumption \eqref{assCoeff}, we still want to quickly comment on one-dimensional regions of existence. The example below illustrates, that knowledge on $\Ex$ alone can not determine the limit---in stark contrast to Theorem \ref{steadyStateLimit}.
    \begin{remark}\textit{{(One-dimensional $\Ex$)}}\label{steadyStateOneDE}\\
        If there are no steady states with positive Type I and Type II mass, the situation is more delicate.
Consider for example  $Q\rs = \binom{r+s}{r} e^{-(r+s)^{\frac 2 3}} + \delta_{s = 2r} e^{r^2}.$ Then $\Ex = [0,1] \times \{0\} \cup \{0\} \times [0,1]$. Now let $z_n,w_n$ be such that 
\[ \rho^n(z_n,w_n) \coloneqq \sumrs[1][n] r z_n^r w_n^s Q\rs = \rho_0 \text{ and } \sigma^n(z_n,w_n) \coloneqq \sumrs[1][n] s z_n^r w_n^s Q\rs = \sigma_0.\]
In particular we have
\begin{equation*}
  \begin{split}
    \rho^n(z_n,w_n) &= \frac{z_n}{z_n + w_n} \summ[1][n]m(z_n+w_n)^m e^{-m^{\frac 2 3}} + \summ[1][\lfloor n/ 3\rfloor] m z_n^m w_n^{2m} e^{m^2} = \rho_0 \\
    \sigma^n(z_n,w_n) &= \frac{w_n}{z_n + w_n} \summ[1][n]m(z_n+w_n)^m e^{-m^{\frac 2 3}} + 2\summ[1][\lfloor n/ 3 \rfloor] m z_n^m w_n^{2m} e^{m^2} = \sigma_0,
  \end{split}
\end{equation*}
which we can rearrange to 
\[ \left(\frac{z_n-w_n/2}{z_n + w_n}\right) \summ[1][n]m(z_n+w_n)^m e^{-m^{\frac 2 3}} = \rho_0 - \frac{\sigma_0}2.\]
Now, let us assume that $z_n \to z < 1$ and $w_n \to 0$. Then we can take the limit to obtain
\begin{equation*}
  \begin{split}
    \summ m z^m e^{-m^{\frac 2 3}} = \rho_0 - \frac{ \sigma_0} 2.
  \end{split}
\end{equation*}
If $z_n \to 1$ and $w_n \to 0$, we find with Fatou and the fact that $\summ[1][n]m(z_n+w_n)^m e^{-m^{\frac 2 3}} \le \rho_0 + \sigma_0$ that
\begin{equation*}
  \begin{split}
    \summ m e^{-m^{\frac 2 3}} \le \rho_0 - \frac{ \sigma_0} 2.
  \end{split}
\end{equation*}
Analogously we find for $w_n \to w < 1$ and $z_n \to 0$
\begin{equation*}
    \summ m w^me^{-m^{\frac 2 3}} = \sigma_0 - 2\rho_0 
\end{equation*}
and for $w_n \to 1$ with $z_n \to 0$
\begin{equation*}
    \summ m e^{-m^{\frac 2 3}} \le \sigma_0 - 2\rho_0.
\end{equation*}
Combining this with the positivity of $w_n,z_n$ we have completely characterised the limiting behaviour. Strikingly, the Type II (or Type I) mass escaping to infinity takes with it some of the Type I (or Type II) mass. In particular, analysing only the one-dimensional steady states in $\Ex$ cannot possibly capture the limiting behaviour.\\
Of course, the Type II mass can also escape independent of the Type I mass, which will happen if we change the Dirac in the example above to $\delta_{r=0}$. Then, we are left with the classical situation, where $z_n \to z$ with either $\rho(z,0) = \rho_0$ or $z=1$.

    \end{remark}

\section{Relative Entropies}
Now that we have understood the steady states, we can study the relative entropy of any $c\in X^+$ with respect to some steady state. It turns out, that they are just $\Vzw + \sumrs \cinfzw\rs + \rho(c) \ln \frac 1 z + \sigma(c) \ln \frac 1 w$, which means solutions to \eqref{mcBD} decrease them over time. As mentioned in the introduction, in the one-component system, solutions to \eqref{ocBD} are minimising sequences of the relative entropy. In anticipation that this may carry over to the two-component system, we start with a general study of the relative entropies. We can determine the minimum of the relative entropies over all sequences with fixed masses $\rho,\sigma$ and characterise the minimising sequences. We will finish this section, by considering the relative entropy with respect to the quasi steady state, which proved to be the crucial quantity, when studying the convergence rate to the limit \cite{canizo:EDEstimate,jabin:convergenceRate}.\\
Remember that $\Ex$ denotes the region of existence (Definition \ref{regionOfExistence}) and assume in this Section, that there are $0<z,w$ with $(z,w)\in \Ex$ (which follows from assumption \eqref{assCoeff}).
    \begin{definition}\textit{(Relative entropy)}\label{relativeEntropy}\\
        Let $0<z,w$ with $(z,w)\in\Ex$. We denote the relative entropy of $c\rs \in X^+$ with respect to $\cinfzw$ as
\begin{equation*}
    H[c|\cinfzw] \coloneqq \sumrs \cinfzw\rs \Psi\left(\frac{c\rs}{\cinfzw\rs}\right),
    \text{ where }\Psi(x) \coloneqq x\ln(x) + 1 -x.
\end{equation*}

\end{definition}
Since solutions of \eqref{mcBD} conserve mass, it is sensible to study $H$ under that constraint.
    \begin{lemma}\textit{(Connection between steady states)}\label{relativeEntropyTwoSteadyStates}\\
        Let $c\in X^+_{\rho,\sigma} \coloneqq \Big\{ c\in X^+ \,\Big|\, \sumrs r c\rs = \rho \text{ and } \sumrs s c\rs = \sigma\Big\} $, then
\begin{equation*}
  H[c|\cinfuv] = H[c|\cinfzw] + H[\cinfzw|\cinfuv] + \ln\left( \frac z u\right) (\rho - \rho(\cinfzw)) + \ln \left(\frac w v\right) (\sigma - \sigma(\cinfzw))
\end{equation*}
holds true.

    \end{lemma}
\begin{proof}
  We calculate directly
\begin{equation*}
  \begin{split}
    H[c|\cinfuv] 
    \overset{\text{def. }\Psi} &= \sumrs c\rs \ln\left(\frac {c\rs}{\cinfuv\rs}\right) + \cinfuv\rs - c\rs\\
    &= H[c|\cinfzw] + \sumrs  r c\rs \ln\left(\frac{z}{u}\right) + s c\rs \ln\left(\frac w v \right) +\cinfuv\rs - \cinfzw\rs \\
    \overset{c\in X^+_{\rho,\sigma}} &= H[c|\cinfzw] + \rho \ln\left(\frac{z}{u}\right) + \sigma \ln\left(\frac w v \right) + \sumrs \cinfuv\rs - \cinfzw\rs \\
    &= H[c|\cinfzw] + H[\cinfzw|\cinfuv] + \ln\left( \frac z u\right) (\rho - \rho(\cinfzw)) + \ln \left(\frac w v\right) (\sigma - \sigma(\cinfzw)).\qedhere
  \end{split}
\end{equation*}

\end{proof}
Theorem \ref{steadyStateLimit} allows to determine the minimum of $H$. Remember that $\Psi\ge0$ and hence $H\ge 0.$
    \begin{theorem}\textit{(Minimising the relative entropy)}\label{relativeEntropyMinimum}\\
        For fixed $\rho,\sigma >0$, let $\cinfzw$ be the steady state that satisfies \eqref{steadyStateLimitEq}. Then, we have 
\begin{equation*}
  \inf\limits_{c\in X^+_{\rho,\sigma}} H[c|\cinfzw] = 0.
\end{equation*}

    \end{theorem}
\begin{proof}
  For $N\in\N$ let $z_N$ and $w_N$ be such that 
\begin{equation*}
  \sumrs[1][N] r  z_N^r w_N^s Q\rs = \rho \quad\text{ and }\quad 
  \sumrs[1][N] s  z_N^r w_N^s Q\rs = \sigma,
\end{equation*}
which exists due to Lemma \ref{finiteSteadyStateDB}. We denote $c^N\rs \coloneqq \mathds{1}_{\{r+s\le N\}} z_N^r w_N^s Q\rs$.  By Theorem \ref{steadyStateLimit} we know that $z_N \to z$ and $w_N \to w$ and in particular $c^N\weakstar \cinfzw$. Then we get for sufficiently large $N$
\begin{equation*}
  \begin{split}
    H[c^N|\cinfzw] &= \ln\left(\frac{z_N}{z}\right)\sumrs[1][N] r c^N\rs 
                    + \ln\left(\frac{w_N}{w}\right)\sumrs[1][N] s c^N\rs 
                    -\sumrs[1][N] c^N\rs + \sumrs \cinfzw \\
                   &=\underbrace{\ln\left(\frac{z_N}{z}\right)}_{\to 0}\rho 
                    + \underbrace{\ln\left(\frac{w_N}{w}\right)}_{\to 0}\sigma
                    + \underbrace{\sumrs \cinfzw\rs - c^N\rs}_{\to 0 \text{ by Prop \ref{weakMetric}}} \to 0.
  \end{split}
\end{equation*}

\end{proof}
Together with Lemma \ref{relativeEntropyTwoSteadyStates}, this characterises the minimising sequences.
    \begin{corollary}\textit{(Minimising sequences)}\label{relativeEntropyMinimizingSequences}\\
        For fixed $\rho,\sigma >0$, let $\cinfzw$ be the steady state satisfying (\ref{steadyStateLimitEq}). Any minimising sequence $(c^n)\subset X^+_{\rho,\sigma}$ of any $H[c^n|\cinfuv]$ satisfies $c^n \weakstar \cinfzw$.

    \end{corollary}
\begin{proof}
  By Lemma \ref{relativeEntropyTwoSteadyStates} we know that $c^n$ is also a minimising sequence of $H[c^n|\cinfzw]$. By Theorem \ref{relativeEntropyMinimum}, we know that $H[c^n|\cinfzw] \to 0$. And since $\Psi(x)$ has a unique zero point at $x=1$, we get that $c^n\rs \to \cinfzw\rs$ for all $r,s$.

\end{proof}
In fact, we can go one step further and show that any mass escaping to infinity has to concentrate in regions of $\Omega$, where $\frac 1 {r+s} \ln(\cinfzw\rs) = 0.$ To do this, we essentially copy the duality argument for $\Psi(1+x)$, that was already used in \cite{jabin:convergenceRate}, to derive a Csiszár--Kullback type inequality.
    \begin{corollary}\textit{(Concentration of mass)}\label{concentrationOfMass}\\
        For fixed $\rho,\sigma >0$, let $\cinfzw$ be the steady state satisfying (\ref{steadyStateLimitEq}) and $(c^n)\subset X^+_{\rho,\sigma}$ be a minimising sequence of $H[c^n|\cinfzw]$. Then, for any $M\in \N$ and $\delta>0$, we have as $n\to\infty$
\begin{equation}
  \sumgamma[\Gamma_{M,\delta}] (r+s) |c^n\rs - \cinfzw\rs| \to 0, \text{ where }
  \Gamma_{M,\delta} \coloneqq \left\{(r,s)\in \Omega\,\middle|\, r+s\le M \text{ or }\frac{\ln(\cinfzw\rs)}{r+s} \le - \delta\right\}.
\end{equation}

    \end{corollary}
\begin{proof}
  Define $f(x) \coloneqq \Psi(1+x)$ for $x\ge0$ and then for $y\ge0$ its young complement via
\[f^*(y) \coloneqq \sup\limits_{x\ge 0}\big(x y - f(x)\big) = e^y - 1 - y,\]
so that for any $x,y \ge0$ we have $x y \le f(x) + f^*(y).$ From these definitions, we see for all $x,y\ge 0$ 
\begin{equation}\label{properties:ffdual}
  f(|x-1|) \le \Psi(x)\text{ and } f^*(\eps y) \le \eps^2 f^*(y) \text{ for  any }0\le \eps\le 1.
\end{equation}
If we now set $x\rs = \frac {c\rs^n}{\cinfzw\rs}$, we find for all $0<\eps<1$
\begin{equation}
  \begin{split}
    \sumgamma[\Gamma_{M,\delta}] (r+s) |c^n\rs - \cinfzw\rs|
    &= \frac{2}{\eps\delta}\sumgamma[\Gamma_{M,\delta}] \frac{\eps \delta} 2 (r+s) \cinfzw\rs |x\rs - 1|\\
  \overset{\text{young}}&\le \frac{2}{\eps\delta}\sumgamma[\Gamma_{M,\delta}]\cinfzw\rs f(|x\rs - 1|) + \cinfzw\rs f^*\Big(\frac{\eps \delta} 2 (r+s)\Big)\\
  \overset{\text{\eqref{properties:ffdual}}}&\le \frac 2 {\eps \delta} \underbrace{H[c^n|\cinfzw]}_{\mathclap{\to 0 \text{ by Theorem \ref{relativeEntropyMinimum}}}} + \frac {2\eps}{\delta} \underbrace{\sumgamma[\Gamma_{M,\delta}] \cinfzw\rs f^*\Big(\frac \delta 2 (r+s)\Big)}_{\eqqcolon (*)}.
  \end{split}
\end{equation}
If $(*)$ is bounded we can make the right hand side arbitrarily small by first taking $\eps$ small and then $n$ large. To bound $(*)$, note that $f^*(y) \le e^y$ and hence by the definition of $\Gamma_{M,\delta}$
\[\sumgamma[\Gamma_{M,\delta}] \cinfzw\rs f^*\Big(\frac \delta 2 (r+s)\Big)
\le \sumrs[1][M] \cinfzw\rs f^*\Big(\frac \delta 2 (r+s)\Big) + \sumrs e^{-\frac {\delta}2 (r+s)}  < \infty.\]

\end{proof}
Also, we can determine what the minimal relative entropy for any $c$ with given masses $\rho,\sigma$ is.
    \begin{corollary}\textit{(Minimal relative entropy)}\label{relativeEntropyMinimal}\\
        For fixed $\rho,\sigma >0$, let $\cinfzw$ be the steady state satisfying (\ref{steadyStateLimitEq}). Then, we have for any $c \in X^+_{\rho,\sigma}$
\begin{equation*}
  H[c|\cinfuv] \ge H[c|\cinfzw] + H[\cinfzw|\cinfuv] \ge H[c|\cinfzw].
\end{equation*}

    \end{corollary}
\begin{proof}
  The claim is a direct consequence of Lemma \ref{relativeEntropyTwoSteadyStates} combined with (\ref{steadyStateLimitEq}).

\end{proof}
As seen in \cite{jabin:convergenceRate} and later in \cite{canizo:EDEstimate}, the crucial quantity for an entropy -- entropy dissipation estimate is the relative entropy with respect to the quasi steady state $\cbar\rs \coloneqq \monr^r \mons^s Q\rs,$ where the monomer densities belong to a solution $c(t)$ of \eqref{mcBD}. Since there is no reason for $(\monr,\mons)$ to be in $\Ex$, $H[c|\cbar]$ is generally not well defined. But on finite subsets $\Gamma\subset\Omega,$ this is not a problem.
\[\text{ For }\Gamma \subset \Omega \text{ finite, we denote }H^\Gamma[c|\cbar] \coloneqq \sumgamma c\rs \Psi\left(\frac {c\rs}{\cbar\rs}\right).\]
With this, we can  generalise Corollary \ref{relativeEntropyMinimal} to the quasi steady state. In the following Lemma, we exploit the weak* semicontinuity of $\sumrs c\rs \ln(\cbar\rs)$ in the regions $\Gamma \subset \Omega$ where $\cbar\rs$ is large. We will use the notation $\cbargamma(\rho,\sigma)$ to be the finite steady state found in Lemma \ref{finiteSteadyStateDB}.
    \begin{lemma}\textit{(Semicontinuity)}\label{relativeEntropySemicontinuity}\\
        Fix $\rho,\sigma >0$ and assume \eqref{assCoeff}. Then, for any $\eps>0$, we can find $M<\infty$ and $\delta_1,\delta_2>0$, such that for any $N$ and 
\[ \Gamma \coloneqq \left\{ (r,s) \in \Omega \, \middle| \, r+s \le M \text{ or }  r+s \le N \text{ with } \ln\Big(\cbar\rs^{\frac 1 {r+s}}\Big) < -\delta_1\right\}\]
we have for any $c \in X^+_{\rho,\sigma}$ with $c_{1,0},c_{0,1} > 0$ and $\sumrs[N+1] (r+s) c\rs \le \delta_2$
\begin{equation*}
  H^\Gamma[c|\cbar] \ge H[c|\cinfzw] + H^\Gamma[\cbargamma(\rho^N,\sigma^N)|\cbar] -\eps,
\end{equation*}
where $\rho^N = \sumrs[1][N] r c\rs$ and $\sigma^N = \sumrs[1][N] s c\rs$ and $\cinfzw$ is the steady state satisfying \eqref{steadyStateLimitEq}.

    \end{lemma}
\begin{proof}
  First, it holds
\begin{equation*}
  \begin{split}
    H^\Gamma[c|\cbar] &= \sumgamma c\rs \ln c\rs - c\rs  + \cbar\rs - c\rs \ln\cbar\rs\\
    &=  \sumgamma c\rs \ln c\rs - c\rs  + \cbar\rs -\sumrs[1][N] c\rs \ln\cbar\rs + \underbrace{\sum\limits_{\substack{(r,s) \in \Gamma^c \\ r+s \le N}} (r+s) c\rs \ln\left(\cbar\rs^{\frac 1 {r+s}}\right)}_{\ge -(\rho+\sigma)\delta_1 \ge -\eps}.
  \end{split}
\end{equation*}
Note, that we have already seen in the proof of Lemma \ref{entropyStrongCont}, that 
\[ \sumrs[M] c\rs \ln c\rs - c\rs \to 0 \text{ as }M\to\infty \text{ independent of } c\rs.\]
Therefore, we continue with 
\begin{equation*}
  \begin{split}
    &\sumgamma c\rs \ln c\rs - c\rs  + \cbar\rs -\sumrs[1][N] c\rs \ln\cbar\rs \\
    &\quad=  H[c|\cinfzw] - \underbrace{\sumgamma[\Gamma^c] c\rs\ln c\rs - c\rs}_{|\cdot| \le \eps \text{ via }M} + \sumrs c\rs \ln \cinfzw\rs -\cinfzw\rs + \sumgamma \cbar\rs - \sumrs[1][N] c\rs\ln\cbar\rs\\
    &\quad\ge -\eps + H[c|\cinfzw] + \sumrs[1][N] c\rs \ln\left(\frac{\cinfzw\rs}{\cbar\rs}\right)
    +\sumrs[N+1](r+s) c\rs \underbrace{\ln\left(\sqrt[r+s]{\cinfzw\rs}\right)}_{|\cdot| \le C \text{ by (\ref{assCoeff})}} -  \sumrs \cinfzw\rs + \sumgamma \cbar\rs\\
    &\quad\ge -\eps + H[c|\cinfzw] + \sumrs[1][N] c\rs \ln\left(\frac{\cinfzw\rs}{\cbar\rs}\right) -\delta_2 C -  \sumrs \cinfzw\rs + \sumgamma \cbar\rs.
  \end{split}
\end{equation*}
If we chose $\delta_2$ so small, that $C\delta_2\le \eps$ and use ${\frac{\cinfzw\rs}{\cbar\rs} = \big(\frac z {c_{1,0}}\big)^r \big(\frac w {c_{0,1}}\big)^s}$ , we find for ${\cbargamma = \cbargamma(\rho^N,\sigma^N)}$
\begin{equation*}
  \begin{split}
  H^\Gamma[c|\cbar] &\ge -3\eps + H[c|\cinfzw] + \rho^N \ln \Big(\frac z {c_{1,0}}\Big) + \sigma^N \ln \Big(\frac w {c_{0,1}}\Big) -  \sumrs \cinfzw\rs + \sumgamma \cbar\rs\\
    &= -3\eps + H[c|\cinfzw] + \sumgamma \cbargamma\rs \ln\Big(\frac{\cbargamma}{\cbar}\Big) + \rho^N \ln \Big(\frac z {\cbargamma_{1,0}}\Big) + \sigma^N \ln \Big(\frac w {\cbargamma_{0,1}}\Big)
    -  \sumrs \cinfzw\rs + \sumgamma \cbar\rs\\
    &= -3\eps + H[c|\cinfzw] + H^\Gamma[\cbargamma|\cbar] + \rho^N \ln \Big(\frac z {\cbargamma_{1,0}}\Big) + \sigma^N \ln \Big(\frac w {\cbargamma_{0,1}}\Big) -  \sumrs \cinfzw\rs + \sumgamma \cbargamma\rs\\
    &= -3\eps + H[c|\cinfzw] + H^\Gamma[\cbargamma|\cbar] + H^\Gamma[\cinfzw|\cbargamma]  \\
    &\qquad+ \ln \Big(\frac z {\cbargamma_{1,0}}\Big)(\rho^N - \rho^\Gamma(\cinfzw)) +  \ln \Big(\frac w {\cbargamma_{0,1}}\Big)(\sigma^N - \sigma^\Gamma(\cinfzw))-  \underbrace{\sumgamma[\Gamma^c] \cinfzw\rs }_{|\cdot| \le \eps \text{ via }M}.
  \end{split}
\end{equation*}
So it remains to show, that $\cbargamma_{1,0}$ is close to $z$ and $\cbargamma_{0,1}$ close to $w$. To do so, let $z_\Gamma,w_\Gamma$, be such that $\sumgamma r z_\Gamma^r w_\Gamma^s Q\rs = \rho$ and $\sumgamma s z_\Gamma^r w_\Gamma^s Q\rs = \sigma$. From Theorem \ref{steadyStateLimit}, we know, that $z_\Gamma \to z$ and $w_\Gamma \to w$, where the distance can be made arbitrarily small, by choosing $M$.
Finally, we know
\newcommand{\propzwn}{z_\Gamma^r w_\Gamma^s Q\rs}
\begin{equation*}
  H^\Gamma[\cbargamma|\propzwn] + H^\Gamma[\propzwn|\cbargamma] = \ln\left(\frac {\cbargamma_{1,0}}{z_\Gamma}\right) (\rho^N - \rho) +  \ln\left(\frac {\cbargamma_{0,1}}{w_\Gamma}\right) (\sigma^N - \sigma).
\end{equation*}
Since $\cbargamma_{1,0},\cbargamma_{0,1},z_\Gamma,w_\Gamma$ are bounded from above and below (as seen in the proof of Theorem \ref{steadyStateLimit}), the right hand side is arbitrarily small depending on $\delta_2$. But then the monomer terms on the left hand side are small, and in particular $|\cbargamma_{1,0} - z_\Gamma|$ and $|\cbargamma_{0,1} -w_\Gamma|$ are small.

\end{proof}

\section{Summary}
\subsection{Differences Between the One- and Two-Component Systems}\label{subsectionCoefficients}

In the general theory of Section \ref{basicProperties} the main difference is the difficulty of establishing the weak formulation \eqref{mcBDWeak} rigorously.
For the one-component system, any sequence $h_m$, for which $|h_m-h_{m-1}|$ is bounded, can be used as a test sequence. We cannot establish the analogous result for the two-component system. As we have seen, he problem is the $r+s = N$ boundary term in \eqref{weakmcBDfinite}. Generally speaking, our methods only allow us to prove the weak formulation for one-dimensional test sequences, i.e. $g\rs = h_r$ or $g\rs = h_s$ or $g\rs = h_{r+s}$ for $r+s$ large---not a satisfactory result, which is why we did not state it rigorously. The same problem weakened our uniqueness result (Theorem \ref{mcBDUniqueness}) and stopped us from establishing the entropy equation for any initial data (Theorem \ref{entropyEquation}). However, if one manages to establish a better uniqueness result, then Theorem \ref{entropyInequality} will still provide $\odv{}{t} V(c(t)) \le -D.$

In the two-component system, equilibria are parameterised over a two-dimensional manifold $\Ex$.  We have seen, that $\Ex$ can have very different shapes, depending on $Q\rs$ (Example \ref{steadyStateDBExample}). Nonetheless, we can describe all steady states as limits of finite steady states, by replacing the monotonicity structure of the one-component system $z \mapsto \rho(\cbar^z)$ by an appropriate convexity structure. This proved to be a key step in understanding the minimisation of the relative entropy $H[c|\cinfzw]$ over $X_{\rho,\sigma}^+$ (positive sequences with fixed Type I and Type II masses).

\subsection{The Problem of Long Time Behaviour}

From the entropy equation we may deduce, that any solution $c(t)$ will be close to the set of equilibria as $t\to\infty.$ However, proving that $c(t) \weakstar \cinfzw$ for some fixed $(z,w) \in \Ex$, as well as determining $(z,w)$ from the initial condition is considerably more difficult in the two-component system. The corresponding result in the one-component system \cite{penrose:foundations} heavily relies on the weak* continuity of $H[c|\cbar^{z_s}].$ But in the two-component system, $H[c|\cinfzw]$ is generally not weak* continuous for any $(z,w) \in \Ex$. Consider for example $Q\rs = \lambda^r \mu^s \binom{r+s}{r}e^{-(r+s)^{\frac 2 3}}$ from Example \ref{steadyStateDBExample} and let $(z,w) \in \Ex$, i.e. $\lambda z + \mu w \le 1.$ We then find 
\[ \sumrs c\rs \ln \cinfzw\rs = \sumrs (r+s) c\rs \ln\left( \sqrt[r+s]{\cinfzw\rs}\right).\]
We will now argue, that there are no $(z,w)\in \Ex$, so that $\sqrt[r+s]{\cinfzw\rs}\to 1$ as $r+s\to \infty,$ which shows that this term is not weak* continuous. We have 
\begin{equation}
  \begin{split}
    \cinfzw\rs = \exp\left( r\ln(\lambda z) + s \ln(\mu w) + \sum\limits_{i=1}^{r+s} \ln i - \sum\limits_{i=1}^{r} \ln i - \sum\limits_{i=1}^{s} \ln i  + \smallo(r+s)\right).
  \end{split}
\end{equation}
If we approximate $\sum\limits_{i=1}^{n} \ln i \approx \int_1^n \ln(x)\dd x =n(\ln n -1) +1$, we arrive at 
\[ \sqrt[r+s]{\cinfzw\rs} \approx \exp\left(\frac{r}{r+s} \ln(\lambda z) + \frac{s}{r+s}\ln(\mu w) - \frac{r}{r+s}\ln\left(\frac r {r+s}\right) - \frac{s}{r+s}\ln\left(\frac{s}{r+s}\right) \right).\]
But the function,
\[\xi \ln\left(\frac{\lambda z}\xi\right) + (1-\xi) \ln\left(\frac{\mu w}{(1-\xi)}\right) \not\equiv 0\]
is not uniformly $0$ on $\xi \in (0,1)$ for any $z,w$. Therefore, a new approach to determine the long time behaviour is needed. Furthermore, the maximum principle used to determine the exact equilibrium in the one-component case \cite{ball:refinedMaxPrinciple} essentially exploits the sign of the fluxes in the subcritical regime. In the two-component system, even if $\monr$ is small, $\mons$ may be large enough to ensure $(\monr,\mons) \not\in \Ex$ (as is clear from Figure \ref{figure:examplesDBE}), which can introduce a sign change in the fluxes. So there are further hurdles to overcome. We present a solution to these problems in \cite{paperLongTime}.

  \section*{Acknowledgements}
\addcontentsline{toc}{section}{Acknowledgements}
The author gratefully acknowledges the financial support of the Deutsche Forschungsgemeinschaft (DFG, German Research Foundation) through the collaborative research centre ``Analysis of criticality: from complex phenomena to models and estimates'' (CRC 1720, Project-ID 539309657) and the Bonn International Graduate School of Mathematics at the Hausdorff Center for Mathematics (EXC 2047/2, Project-ID 390685813).\\
I would like to thank Barbara Niethammer for numerous helpful discussions and guidance throughout this project.

  \printbibliography
\end{document}